\documentclass[a4paper,12pt]{article}
\usepackage{amssymb,amsmath,amsthm,times,mathrsfs,fullpage}
\usepackage[pdftex]{hyperref}
\hypersetup{plainpages=True, pdfstartview=FitH, bookmarksopen=true,
pdfpagemode=none, colorlinks=true,linkcolor=blue,citecolor=blue}
\usepackage{tikz,pgflibraryshapes,pgf}
\usepackage{color, colortbl}
\usetikzlibrary{shapes,arrows,shadows,snakes}%

\def\pf{\noindent \emph{Proof.}\ }
\def\qed{{\quad\rule{1mm}{3mm}\,}}
\def\le{\leqslant}
\def\ge{\geqslant}

\begin{document}
%\pagecolor{orange!15}

\newtheorem{thm}{Theorem}
\renewcommand*{\thethm}{\Alph{thm}}
\newtheorem{cor}{Corollary}
\newtheorem{lmm}{Lemma}
\newtheorem{conj}{Conjecture}
\newtheorem{pro}{Proposition}
\renewcommand*{\thepro}{\Alph{pro}}
\newtheorem{df}{Definition}
\theoremstyle{remark}\newtheorem{Rem}{Remark}
\newcommand{\bigcell}[2]{\begin{tabular}{@{}#1@{}}#2\end{tabular}}

% \title{\textbf{Dependence between Size and External Path-Length in
% Random Tries}}
\title{\textbf{Dependence between path-length and size in
random digital trees}}
\author{
Michael Fuchs\\
    Department of Applied Mathematics\\
    National Chiao Tung University\\
    Hsinchu 300\\
    Taiwan \and
Hsien-Kuei Hwang\\
    Institute of Statistical Science\\
    Academia Sinica\\
    Taipei 115\\
    Taiwan
}
\date{\today}
\maketitle

%\color{green!20}
%\pagecolor{blue!15}

\begin{abstract}

We study the size and the external path length of random tries and
show that they are asymptotically independent in the asymmetric case
but strongly dependent with small periodic fluctuations in the
symmetric case. Such an unexpected behavior is in sharp contrast to
the previously known results on random tries that the size is totally
positively correlated to the internal path length and that both tend
to the same normal limit law. These two dependence examples provide
concrete instances of bivariate normal distributions (as limit laws)
whose correlation is $0$, $1$ and periodically oscillating. Moreover,
the same type of behaviors is also clarified for other classes
of digital trees such as bucket digital trees and Patricia tries.

\end{abstract}



\section{Introduction}\label{intro}

Tries are one of the most fundamental tree-type data structures in
computer algorithms; see Knuth \cite{Knuth98} and Mahmoud
\cite{Mahmoud92} for a general introduction. Their general efficiency
depends on several shape parameters, the principal ones including the
depth, the height, the size, the internal path-length (IPL), and the
external path-length (EPL); see below for a more precise description
of those studied in this paper. While most of these measures have
been extensively investigated in the literature, we are concerned
here with the question: \emph{how does the EPL depend on the size in
a random trie?} Surprisingly, while the pair
$(\text{IPL},\text{size})$ is known to have asymptotic correlation
coefficient tending to one and to have the same normal limit law
after each being properly normalized (see \cite{FuHwZa,FuLe}), this
paper aims to show that the pair $(\text{EPL},\text{size})$ exhibits
a completely different behavior depending on the parameter of the
underlying random bits being biased or unbiased. This is a companion
paper to \cite{ChFuHwNe} where we clarified the dependence structure
of another class of search trees in computer algorithms.

Given a sequence of binary strings (or keys), one can construct a
binary trie (very similar to constructing a dictionary of binary
words) as follows. If $n=1$, then the trie consists of a single
root-node holding the sole string; if $n\ge2$, the root is used to
direct the strings into the corresponding subtree: if the first bit
of the input string is $0$ (or $1$), then the string goes to the left
(or right) subtree; strings directed to the same subtree are then
processed recursively in the same manner but instead of splitting
according to the first bit, the second bit of each string is then
used. In this way, a binary dictionary-type tree with two types of
nodes is constructed: external nodes for storing strings and internal
nodes for splitting the strings; see Figure~\ref{fg-trie} for a trie
of seven strings.

\vspace*{0.2cm}
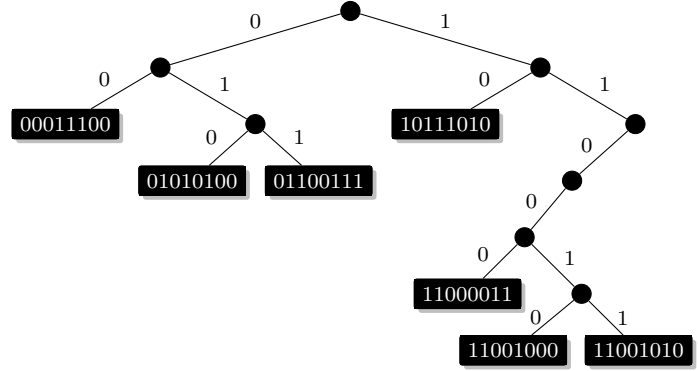
\begin{figure}[!ht]
\begin{center}
\begin{tikzpicture}[level/.style={sibling distance = 5cm/#1,
level distance = 0.75cm}, treenode/.style = {align=center,
inner sep=0pt, text centered}, arn_n/.style = {treenode, circle,
white, draw=black,fill=black, text width=0.6em},
arn_x/.style = {rounded corners=0.2ex,
drop shadow={shadow xshift=0.3ex,shadow yshift=-0.3ex},
fill=white,treenode, rectangle, draw=black,  minimum width=3em,
minimum height=0.7em, font=\scriptsize}]
\node [arn_n] {}
	child{ node [arn_n] {}  % lev 11
		child{ node [arn_x]
		{\colorbox{black}{\textcolor{white}{$00011100$}}}  % x6
			edge from parent
			node[left,yshift=0.15cm,font=\scriptsize] {$0$}
		}
		child{ node [arn_n] {} % lev 22
		child{ node [arn_x]
		{\colorbox{black}{\textcolor{white}{$01010100$}}}% x7
			edge from parent
			node[left,yshift=0.15cm,font=\scriptsize] {$0$}}	
			child{ node [arn_x]
			{\colorbox{black}{\textcolor{white}{$01100111$}}}% x5
			edge from parent
			node[right,yshift=0.15cm,font=\scriptsize] {$1$}}	
			edge from parent
			node[right,yshift=0.15cm,font=\scriptsize] {$1$}
            }
		edge from parent
		node[above,font=\scriptsize] {$0$}
	}
	child{ node [arn_n] {}	% lev 12
            child{ node [arn_x]
			{\colorbox{black}{\textcolor{white}{$10111010$}}} % x3
			edge from parent
			node[left,yshift=0.15cm,font=\scriptsize] {$0$}
            }
            child{ node [arn_n] {}
			child[]{ node [arn_n] {}
				child[]{ node [arn_n] {}
					child[sibling distance = 1.5cm]{ node [arn_x]
					{\colorbox{black}{\textcolor{white}{$11000011$}}}
					edge from parent node[left,yshift=0.10cm,
					font=\scriptsize] {$0$}} %x4
					child[sibling distance = 1.5cm]{ node [arn_n] {}
						child[sibling distance = 1.8cm]{ node [arn_x]
						{\colorbox{black}{\textcolor{white}
						{$11001000$}}}
						edge from parent node[left,yshift=0.00cm,
						font=\scriptsize] {$0$}}	% x2
						child[sibling distance = 1.5cm]{ node [arn_x]
						{\colorbox{black}{\textcolor{white}
						{$11001010$}}} edge from parent node[right,
						yshift=0.00cm,font=\scriptsize] {$1$}}	% x1
					edge from parent node[right,yshift=0.10cm,
					font=\scriptsize] {$1$}
					}
				edge from parent node[left,yshift=0.10cm,
				font=\scriptsize] {$0$}
				}
				child[missing]{ node [arn_n] {}}
				edge from parent node[left,yshift=0.10cm,
				font=\scriptsize] {$0$}
			}
			child[missing]{ node [arn_r] {}}
			edge from parent
			node[right,yshift=0.15cm,font=\scriptsize] {$1$}
            }
		edge from parent
		node[above,font=\scriptsize] {$1$}
	};
\end{tikzpicture}
\caption{\emph{A trie with $n=7$ records: the (filled) circles
represent internal nodes and rectangles holding the binary strings
are external nodes. In this example, $S_n=8$, $K_n=27$, and
$N_n=18$.}}
\label{fg-trie}
\end{center}
\end{figure}

The random trie model we consider here assumes that each of the $n$
binary keys is an infinite sequence of independent Bernoulli bits
each with success probability $0<p<1$. Then the trie constructed from
this sequence is a random trie.

We define three shape parameters in a random trie of $n$ strings:
\begin{itemize}
	\item Size $S_n$: the total number of internal nodes used
          (the circle nodes in Figure~\ref{fg-trie});
	\item IPL (or node path-length, NPL) $N_n$: the sum
	      of the distance between the root and each internal node;
    \item EPL (or key path-length, KPL) $K_n$: the sum
	      of the distance between the root and each external node.
\end{itemize}
We will use mostly NPL in place of IPL, and KPL in place of EPL, the
reason being an easier comparison with the corresponding results
derived for random $m$-ary search trees in the companion paper
\cite{ChFuHwNe}; see below for more details.

By the recursive definition and our model assumption, we have the
following recurrence relations
\begin{equation}\label{dist-rec}
    \left\{
	\begin{split}
		S_n&\stackrel{d}{=}S_{B_n}+S_{n-B_n}^{*}+1,\\
		K_n&\stackrel{d}{=}K_{B_n}+K_{n-B_n}^{*}+n,\\
		N_n&\stackrel{d}{=}N_{B_n}+N_{n-B_n}^{*}
		+S_{B_n}+S_{n-B_n}^{*},
	\end{split}\qquad(n\ge2), \right.
\end{equation}
with the initial conditions $S_n=K_n=N_n=0$ for $n\le 1$, where
$B_n=\text{Binom}(n,p)$ denotes a binomial distribution with
parameters $n$ and $p\in(0,1)$. Also $(S_n^*), (K_n^*)$, and
$(N_n^*)$ are independent copies of $(S_n), (K_n)$ and $(N_n)$,
respectively. While many stochastic properties of these random
variables are known (see Cl\'{e}ment et al. \cite{CFV01}, Devroye
\cite{De05} and \cite{FuHwZa} and many references cited there), much
less attention has been paid to their correlation and dependence
structure.

The asymptotic behaviors of the moments of random variables defined
on tries typically depend on the ratio $\frac{\log p}{\log q}$ being
rational or irrational, where $q=1-p$. So we introduce, similar to
\cite{FuHwZa}, the notation
\begin{align}\label{gk}
	\mathscr{F}[g](z)
	=\begin{cases}{
	    \sum_{k\in\mathbb{Z}}g_k
		z^{-\chi_k}},&{ \text{if\ }
		\frac{\log p}{\log q}\in\mathbb{Q}};\\
		g_0,&{
		\text{if\ } \frac{\log p}{\log q}\not\in\mathbb{Q}},
	\end{cases}
\end{align}
where $g_k$ represents a sequence of (Fourier) coefficients and
$\chi_k= \frac{2rk\pi i}{\log p}$ when $\frac{\log p}{\log q}= \frac
rl$ with $r$ and $l$ coprime. In simpler words, $\mathscr{F}[g](z)$
is a periodic function in the rational case, and a constant in the
irrational case. We also use $\mathscr{F}[\cdot](z)$ as a generic
symbol if the exact form of the underlying sequence matters less, and
in this case each occurrence may not represent the same function.

With this notation, the asymptotics of the mean and the variance are
summarized in the following table; see \cite{FuHwZa,JaSz,Mahmoud92}
and the references therein for more information.
\begin{table}[!h]
\begin{center}
\begin{tabular}{|c||c|c|} \hline
	Shape parameters & $\frac1n(\text{mean})\sim$
	& $\frac1n(\text{variance})\sim$  \\ \hline
	Size $S_n$ & $\mathscr{F}[\cdot](n)$
	& $\mathscr{F}[g^{(1)}](n)$ \\ \hline
	NPL $N_n$ & $\frac{\mathbb{E}(S_n)}{n}\cdot \frac{\log n}{h}
	$ & $\frac{\mathbb{V}(S_n)}{n}
	\cdot \frac{(\log n)^2}{h^2}$ \\ \hline
%	\rowcolor{gray!10}
	KPL $K_n$ & $\frac{\log n}{h}+\mathscr{F}[\cdot](n)$
	& \textcolor{red}{
	$\frac{pq\log^2\frac pq}{h^2}$}
	$\cdot\frac{\log n}{h}+\mathscr{F}[g^{(3)}](n)$
	\\ \hline \hline
%	\rowcolor{gray!10}
	Depth $D_n$ & $\mathbb{E}(D_n) =
	\frac{\mathbb{E}(K_n)}{n}$ &
	$\mathbb{V}(D_n) =
	\frac{\mathbb{V}(K_n)}{n}+O(1)$
	\\ \hline
\end{tabular}
\end{center}
\vspace*{-.4cm}
\caption{\emph{Asymptotic patterns of the means and the variances of
the shape parameters discussed in this paper. Here
$\mathscr{F}[\cdot](n)$ differs from one occurrence to another and
$h=-p\log p-q\log q$ denotes the entropy. Expressions for $g^{(1)}_k$
and $g^{(3)}_k$ will be given below. All three random variables $S_n,
N_n, K_n$ are asymptotically normally distributed.}}
\label{tb-SKN}
\end{table}

Note specially that the leading constant
\[
    \lambda=\lambda_p := \frac{pq\log^2\frac pq}{h^3}
	= \frac{(p\log^2p+q\log^2q)-h^2}{h^3}
\]
in the asymptotic approximation to $\mathbb{V}(K_n)$ equals zero when
$p=q$, implying that $\mathbb{V}(K_n)$ is not of order $n\log n$ but
of linear order in the symmetric case. \emph{This change of order can
be regarded as the source property distinguishing between the
dependence and independence of $K_n$ on $S_n$.}

On the other hand, we have the relation $\mathbb{E}(K_n)
=\mathbb{E}(D_n)n$ between the external path length and the depth
$D_n$, which is defined to be the distance between the root and a
randomly chosen external node (each with the same probability).
Furthermore, we also have the asymptotic equivalent $\mathbb{V}(K_n)
\sim \mathbb{V}(D_n)n$ when $p\neq 1/2$ (or $\lambda>0$), and a
central limit theorem for $D_n$; see Devroye \cite{De98}.

From Table~\ref{tb-SKN}, we see roughly that each internal node
contributes $\frac{\log n}h$ to $N_n$, namely, that $N_n\approx
S_n\cdot \frac{\log n}h$. Indeed, it was proved in \cite{FuHwZa} that
the correlation coefficient of $S_n$ and $N_n$ satisfies
\begin{align}\label{rho-Sn-Nn}
    \rho(S_n,N_n)\sim 1\qquad(0<p<1).
\end{align}
Such a linear correlation was further strengthened in \cite{FuLe},
where it was proved that both random variables tend to the \emph{same}
normal limit law $\mathcal{N}_1$ (with zero mean and unit variance)
\[
	\left(\frac{S_n-\mathbb{E}(S_n)}
	{\sqrt{\mathbb{V}(S_n)}},\frac{N_n-\mathbb{E}(N_n)}
	{\sqrt{\mathbb{V}(N_n)}}\right)
	\stackrel{d}{\longrightarrow}
	(\mathcal{N}_1,\mathcal{N}_1),
\]
where $\stackrel{d}{\longrightarrow}$ denotes convergence in
distribution. In terms of the bivariate normal law $\mathcal{N}_2$
(see Tong \cite{To90}), we can write
\[
	\left(\frac{S_n-\mathbb{E}(S_n)}
	{\sqrt{\mathbb{V}(S_n)}},\frac{N_n-\mathbb{E}(N_n)}
	{\sqrt{\mathbb{V}(N_n)}}\right)^{\intercal}
	\stackrel{d}{\longrightarrow} \mathcal{N}_2(0,E_2),
\]
where $E_2=\scriptsize \begin{pmatrix} 1 & 1\\ 1& 1\end{pmatrix}$ is
a singular matrix and $\mathbf{A}^\intercal$ denotes the transpose of
matrix $\mathbf{A}$.

We show that the correlation and dependence of $K_n$ on $S_n$ are
drastically different. We start with their correlation coefficient.

\begin{thm}\label{main-thm-1} The covariance of the number of
internal nodes and KPL in a random trie of $n$ strings satisfies
\[
    \mathrm{Cov}(S_n,K_n) \sim n \mathscr{F}[g^{(2)}](n),
\]
where $g^{(2)}_k$ is given in Proposition \ref{fc-cov} below, and
their correlation coefficient satisfies
\begin{align}\label{rho-Sn-Kn}
	\rho(S_n,K_n)\sim
	\begin{cases}
		0,&\text{if}\ p\ne \frac12\\
		F(n),&\text{if}\ p=\frac12.
	\end{cases}
\end{align}
Here $F(n)=\frac{\mathscr{F}[g^{(2)}](n)}
{\sqrt{\mathscr{F}[g^{(1)}](n) \mathscr{F}[g^{(3)}](n)}}$ is a
periodic function with average value $0.927\cdots$.
\end{thm}
The result \eqref{rho-Sn-Kn} is to be compared with \eqref{rho-Sn-Nn}
(which holds for all $p\in(0,1)$): \emph{the surprising difference
here comes not only from the (common) distinction between $p=\frac12$
and $p\ne \frac12$ but also from the (less expected) intrinsic
asymptotic nature.}

\begin{figure}[!h]
\begin{center}
\hspace*{0 cm}
\begin{tikzpicture}[]
\node[align=center] at (0,0) {
\includegraphics[height=4cm]{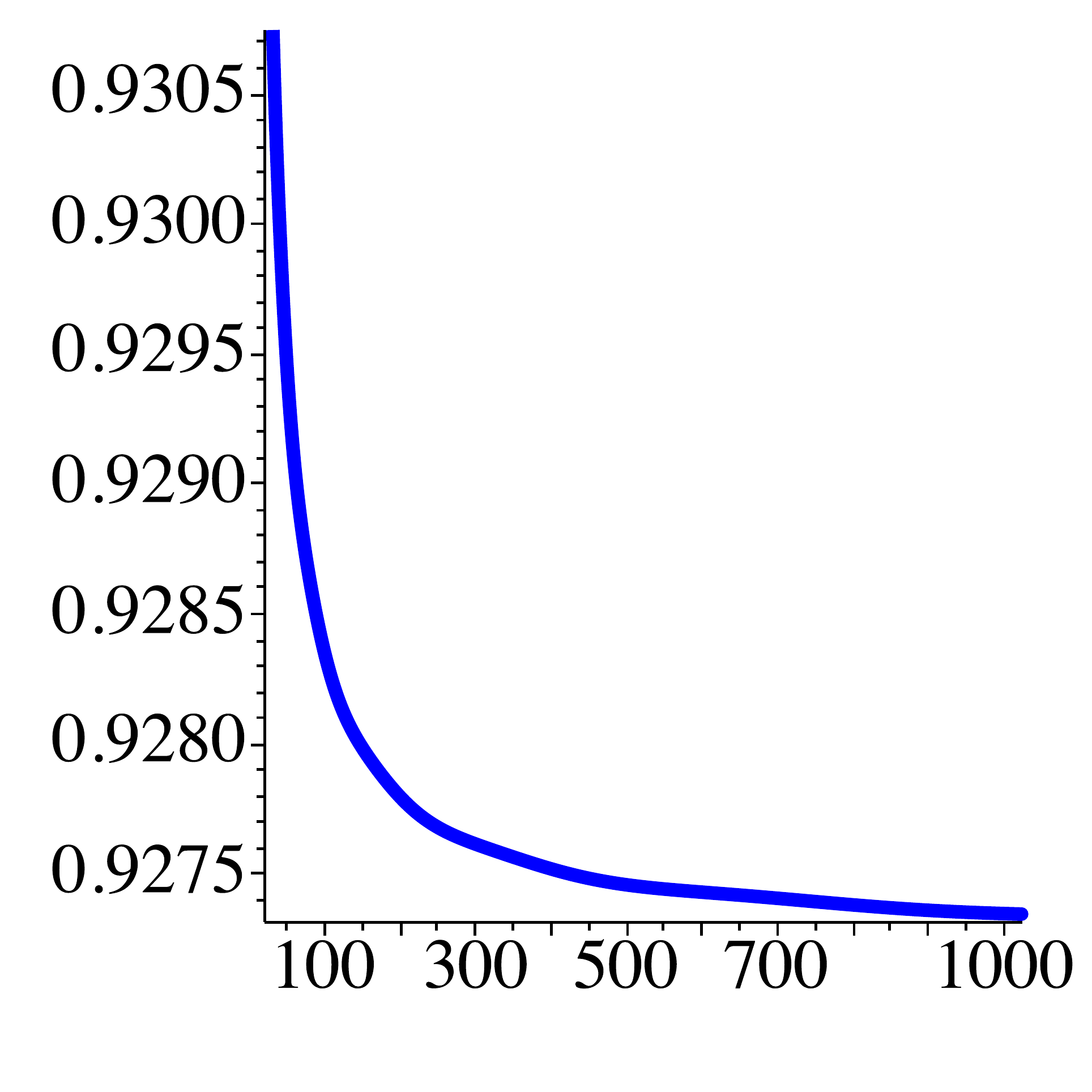}\hspace*{0 pt}
\includegraphics[height=4cm]{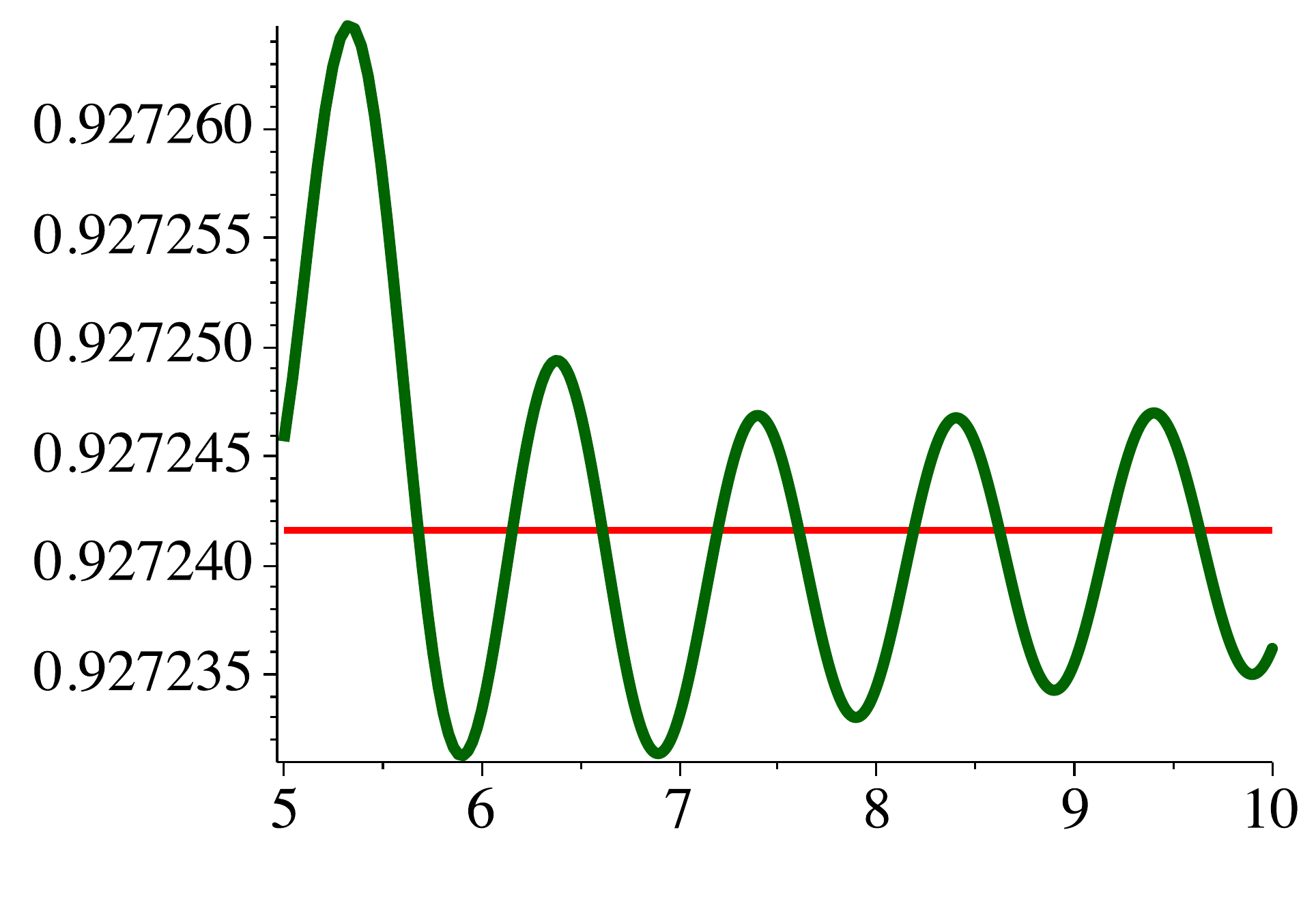}\hspace*{0 pt}
\includegraphics[height=4cm]{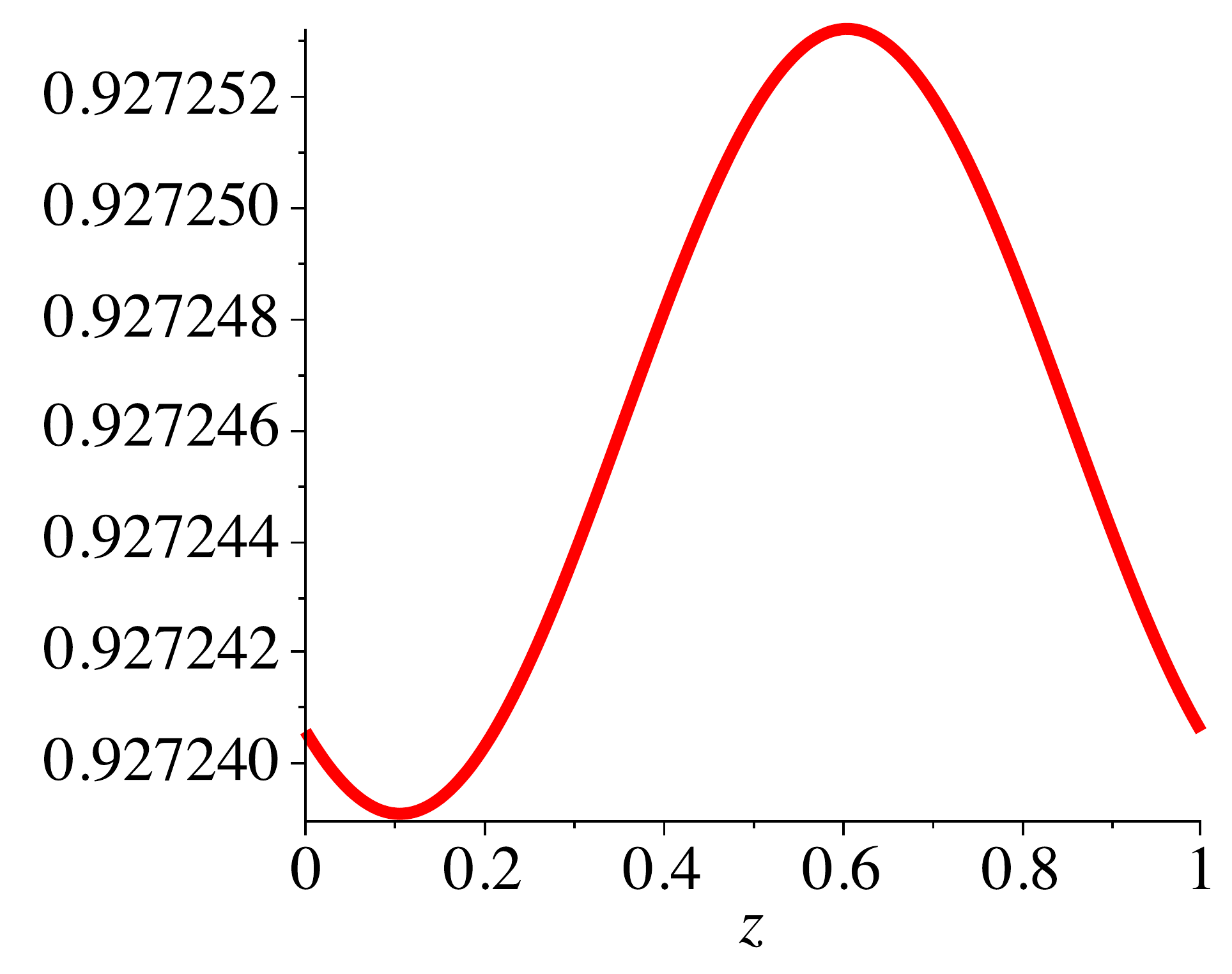}\hspace*{0 pt}
};
\end{tikzpicture}
\end{center}
\vspace*{-.4cm}
\caption{\emph{$p=\frac12$: periodic fluctuations of (i)
$\rho(S_n,K_n)$ (left) for $n=32,\dots, 1024$, (ii)
$\frac{\mathrm{Cov}(S_n,K_n)}{\sqrt{\mathbb{V}(S_n)
(\mathbb{V}(K_n)+1.046)}}$ (middle) in logarithmic scale, and (iii)
$F(n)$ by its Fourier series expansion (right). Note that the
fluctuations are only visible by a proper correction in the
denominator because the amplitude of $F$ is very small:
$|F(\cdot)|\le 1.5\times 10^{-5}$.}}
\end{figure}

Furthermore, we show that this different behavior cannot be ascribed
to the weak measurability of nonlinear dependence of Pearson's
correlation coefficient because the limiting distribution also
exhibits a similar dependence pattern. (For the univariate central
limit theorems implied by the result below, see Jacquet and
R\'{e}gnier \cite{JaRe} where such results were first established.)

\begin{thm}\label{main-thm-2}
\begin{itemize}
\item[(i)] For $p\ne \frac12$, we have
\[
	\left(\frac{S_n-\mathbb{E}(S_n)}
	{\sqrt{\mathbb{V}(S_n)}},\frac{K_n-\mathbb{E}(K_n)}
	{\sqrt{\mathbb{V}(K_n)}}\right)^{\intercal}
	\stackrel{d}{\longrightarrow} \mathcal{N}_2(0,I_2),
\]
where $I_2$ denotes the $2\times 2$ identity matrix.
\item[(ii)] For $p=\frac12$, we have
\[
	\Sigma_n^{-\frac12}
	\begin{pmatrix}
		S_n-\mathbb{E}(S_n) \\
		K_n-\mathbb{E}(K_n)
	\end{pmatrix}
		\stackrel{d}{\longrightarrow}\mathcal{N}_2(0,I_2),
\]
where $\Sigma_n$ denotes the (asymptotic) covariance matrix of $S_n$
and $K_n$:
\[
    \Sigma_n := n\begin{pmatrix}
	    \mathscr{F}[g^{(1)}](n)
		& \mathscr{F}[g^{(2)}](n) \\
		\mathscr{F}[g^{(2)}](n)
		& \mathscr{F}[g^{(3)}](n)
	\end{pmatrix}.
\]
\end{itemize}
\end{thm}
Alternatively, we may define
\[
\Sigma_n:=n\begin{pmatrix}
\mathscr{F}[g^{(1)}](n) & \mathscr{F}[g^{(2)}](n) \\
\mathscr{F}[g^{(2)}](n) & \lambda\log n+\mathscr{F}[g^{(3)}](n)
\end{pmatrix}.
\]
Then both cases can be stated in one as
\[
	\Sigma_n^{-\frac12}
	\begin{pmatrix}
		S_n-\mathbb{E}(S_n) \\
		K_n-\mathbb{E}(K_n)
	\end{pmatrix}
		\stackrel{d}{\longrightarrow}\mathcal{N}_2(0,I_2).
\]
On the other hand, since for bivariate normal distribution, zero
correlation implies independence (see \cite{To90}), it is more
transparent to split the statement into two cases. See
Figure~\ref{fig-SK} for (Monte Carlo) 3D-plots of the joint
distributions of $(S_n,K_n)$ when $n=10^7$.

\begin{figure}[!h]
\begin{center}
\hspace*{0 cm}
\begin{tikzpicture}[]
\node[align=center] at (0,0) {
\includegraphics[width=5.5cm]{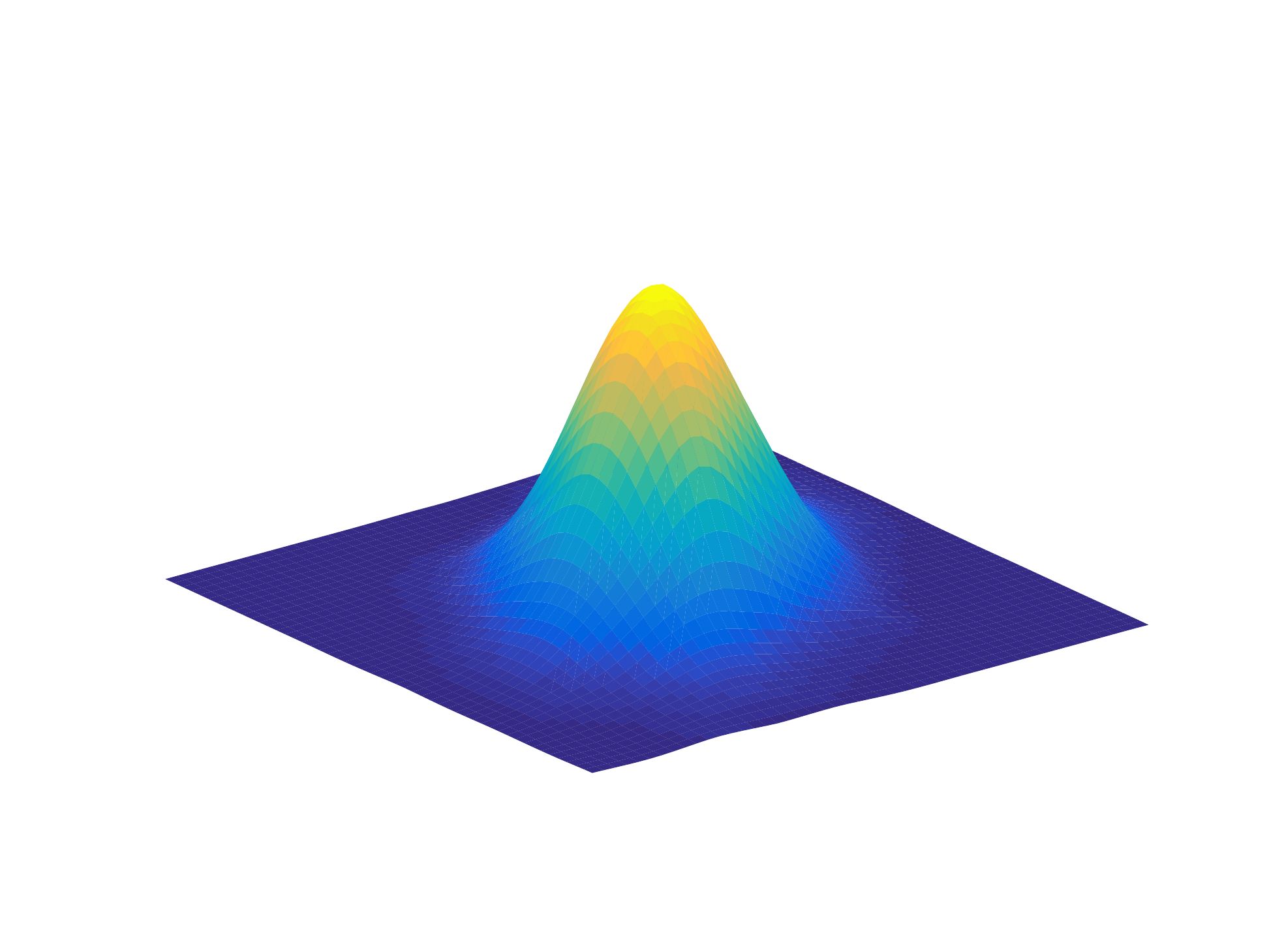}\hspace*{-15pt}
\includegraphics[width=5.5cm]{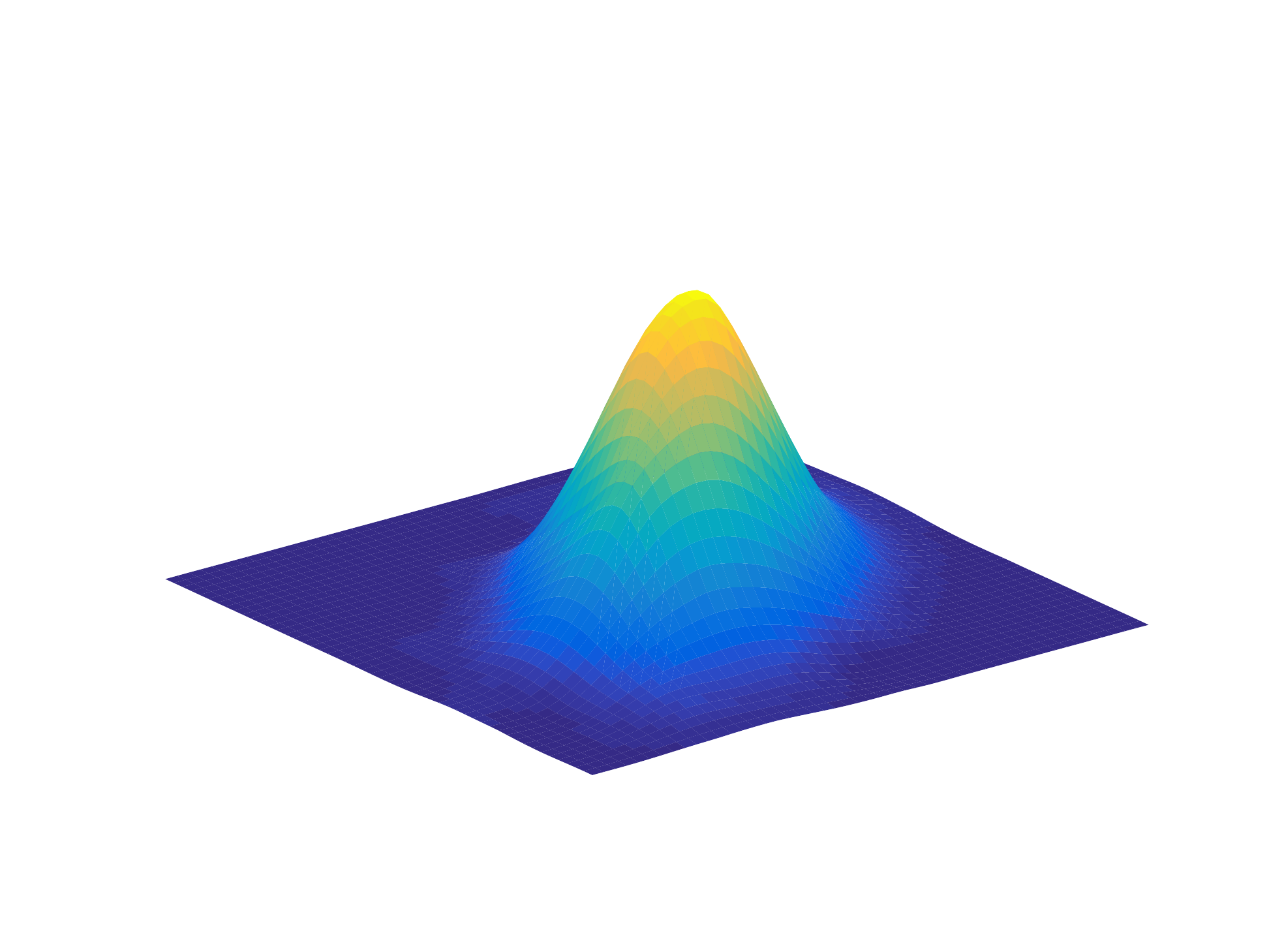}\hspace*{-15pt}
\includegraphics[width=5.5cm]{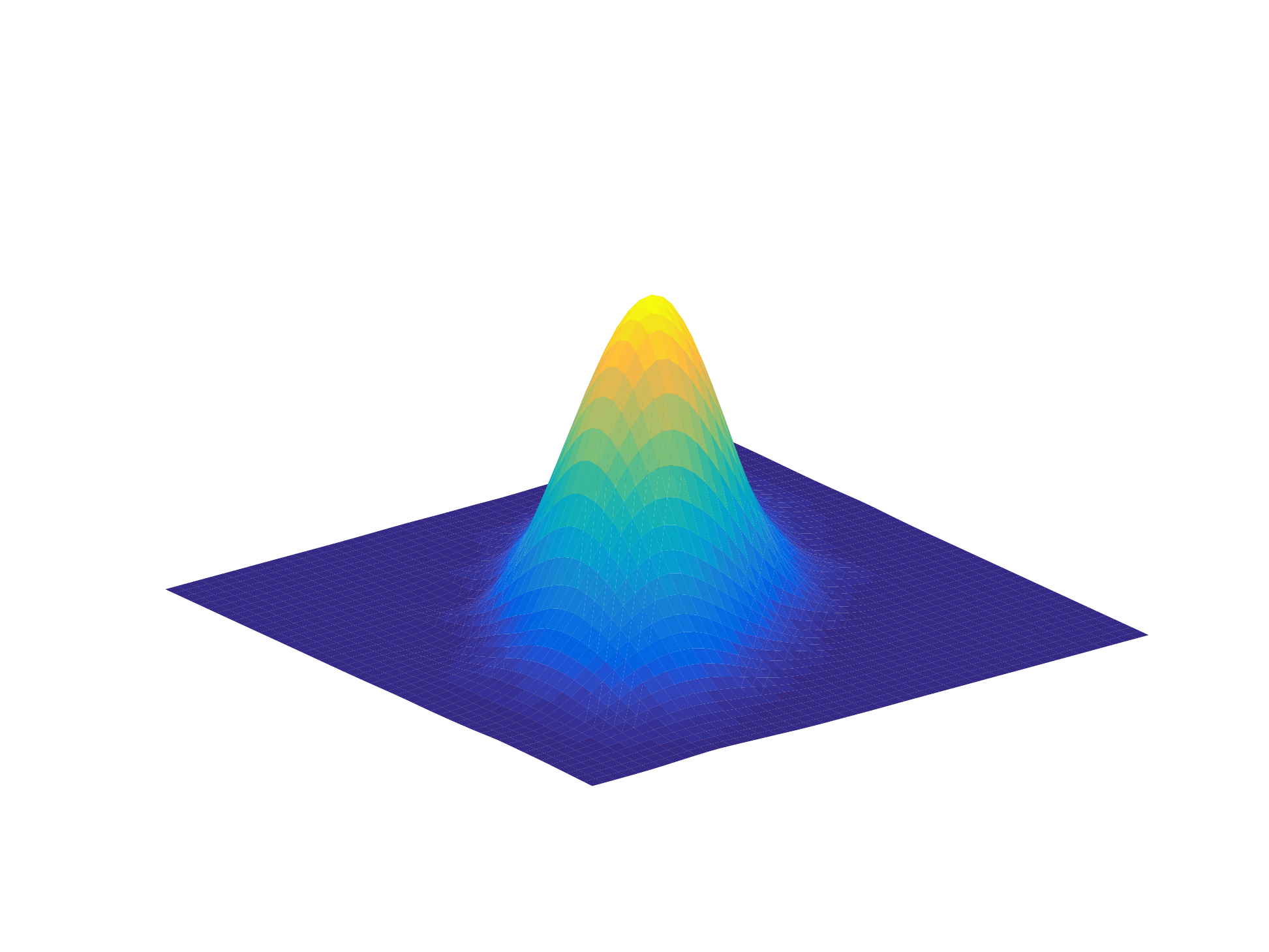}\\
\includegraphics[width=5.5cm]{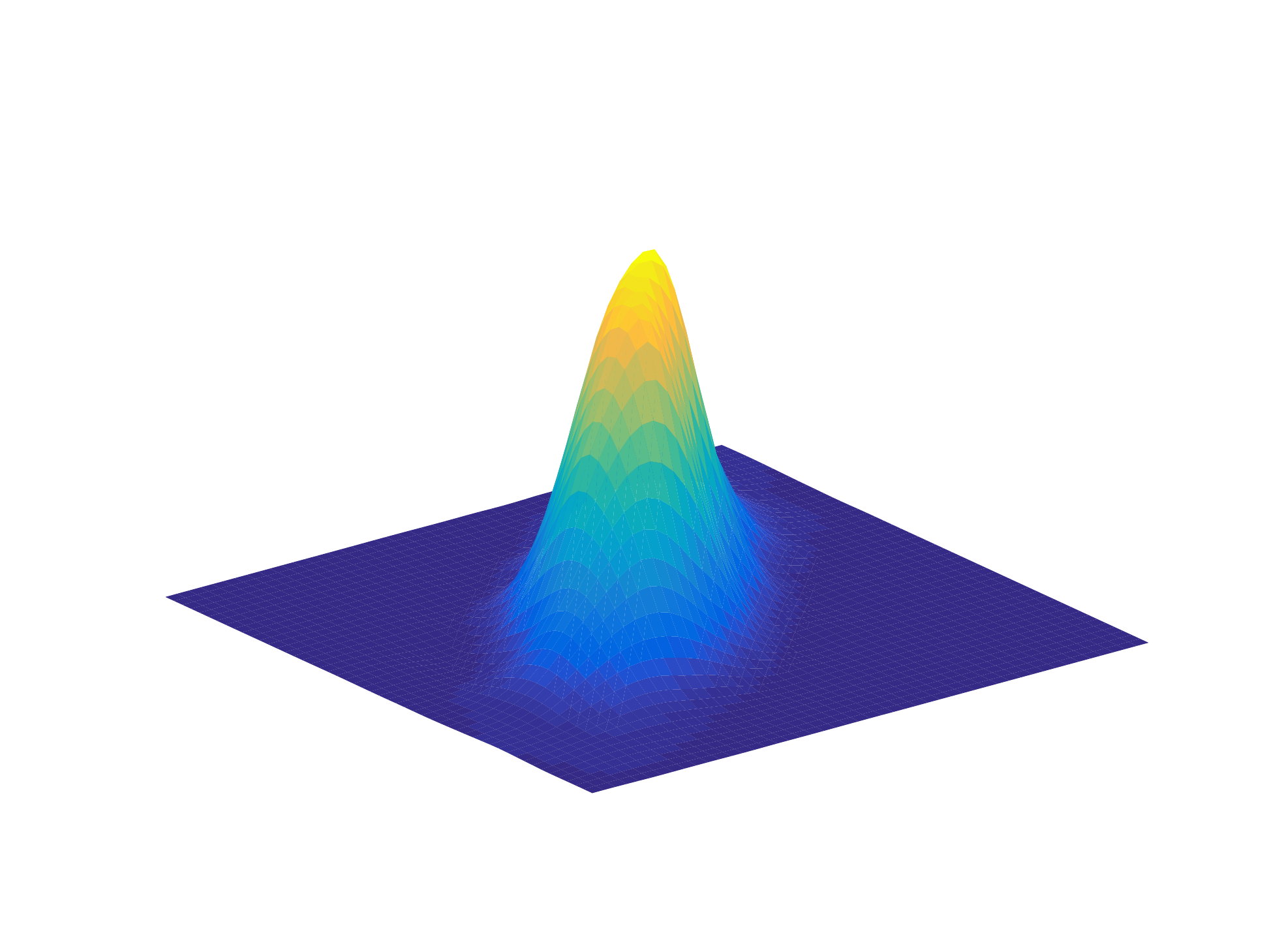}\hspace*{-15pt}
\includegraphics[width=5.5cm]{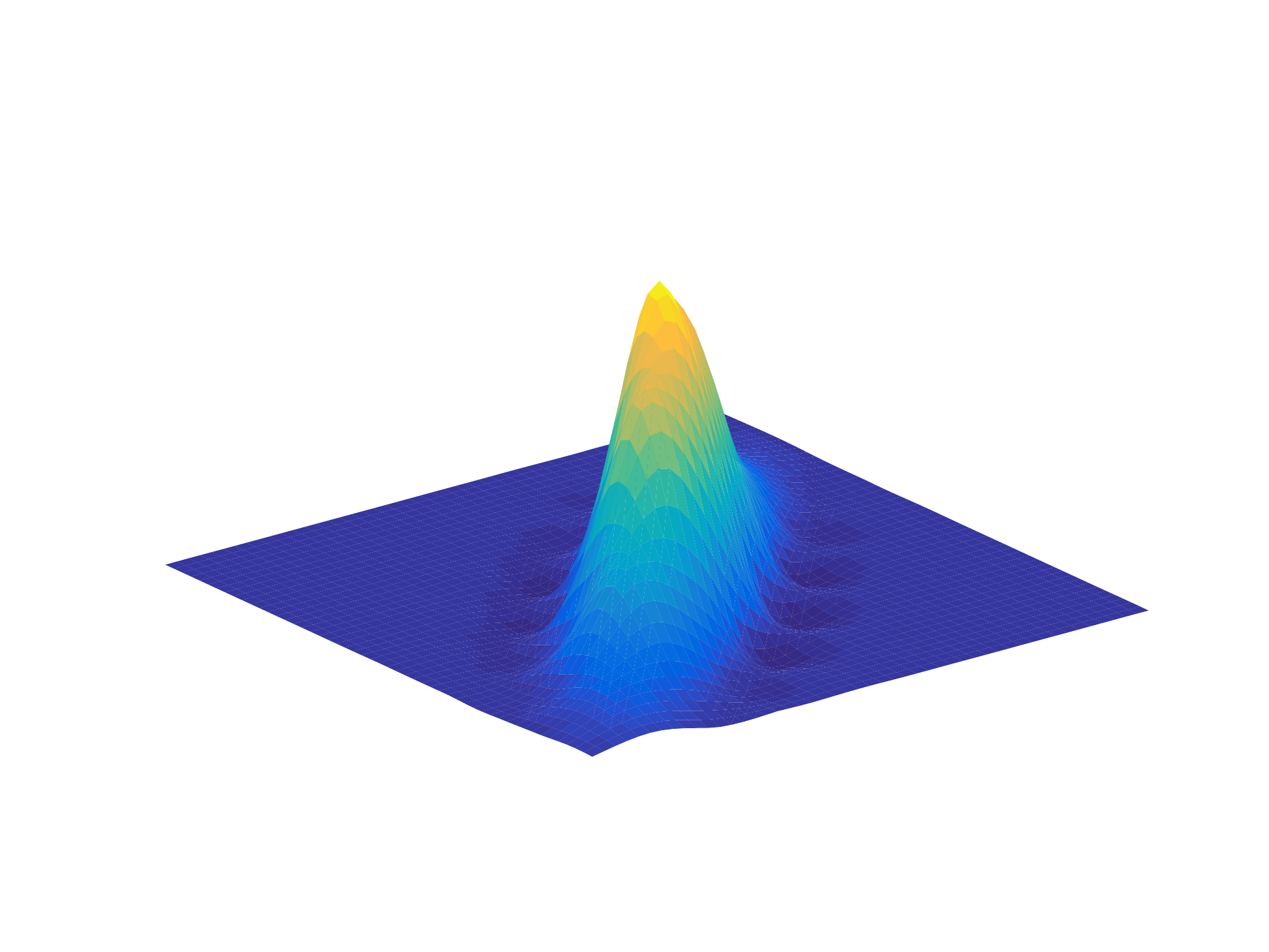}\hspace*{-15pt}
\includegraphics[width=5.5cm]{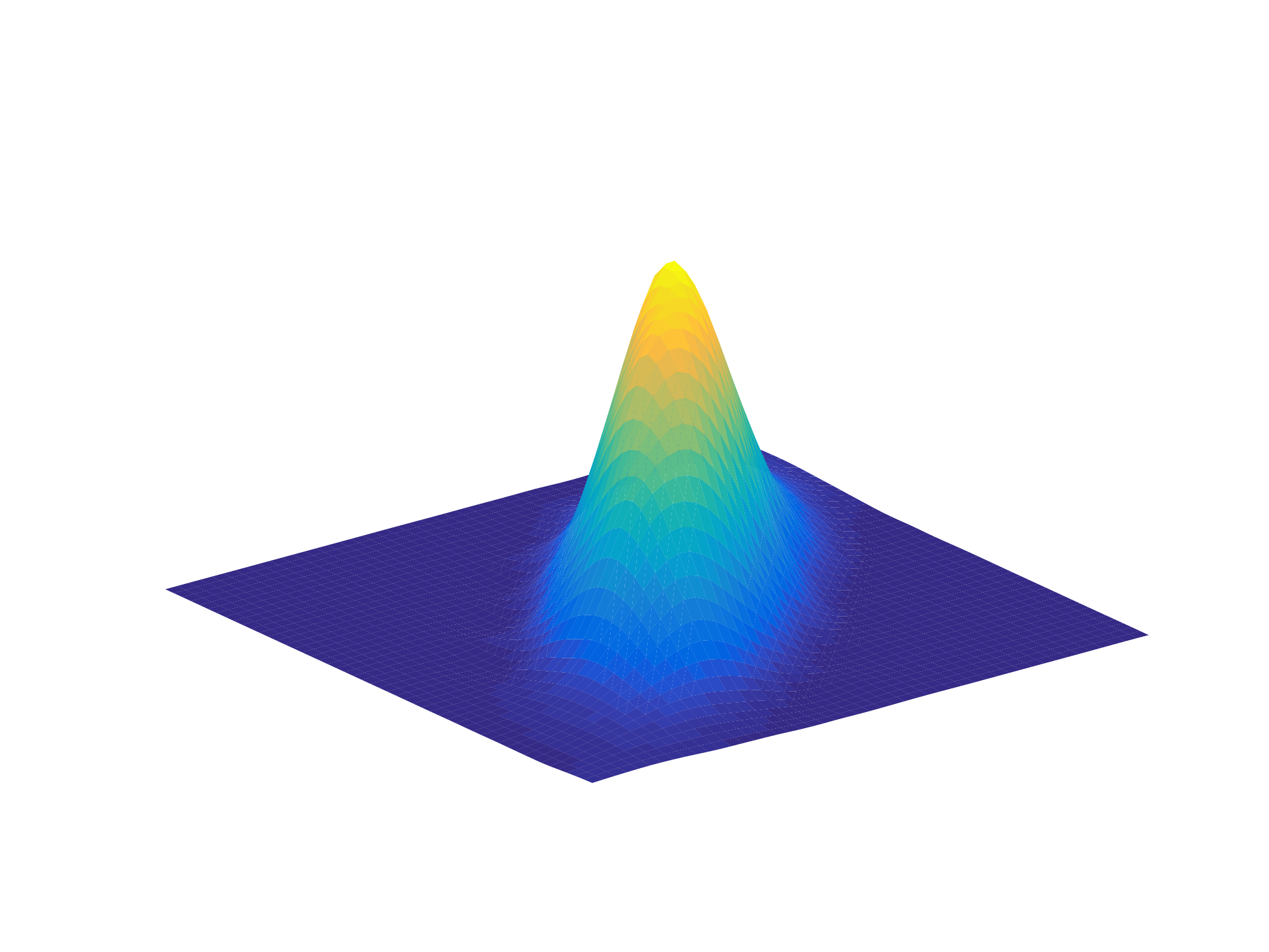}\\
\includegraphics[width=5.5cm]{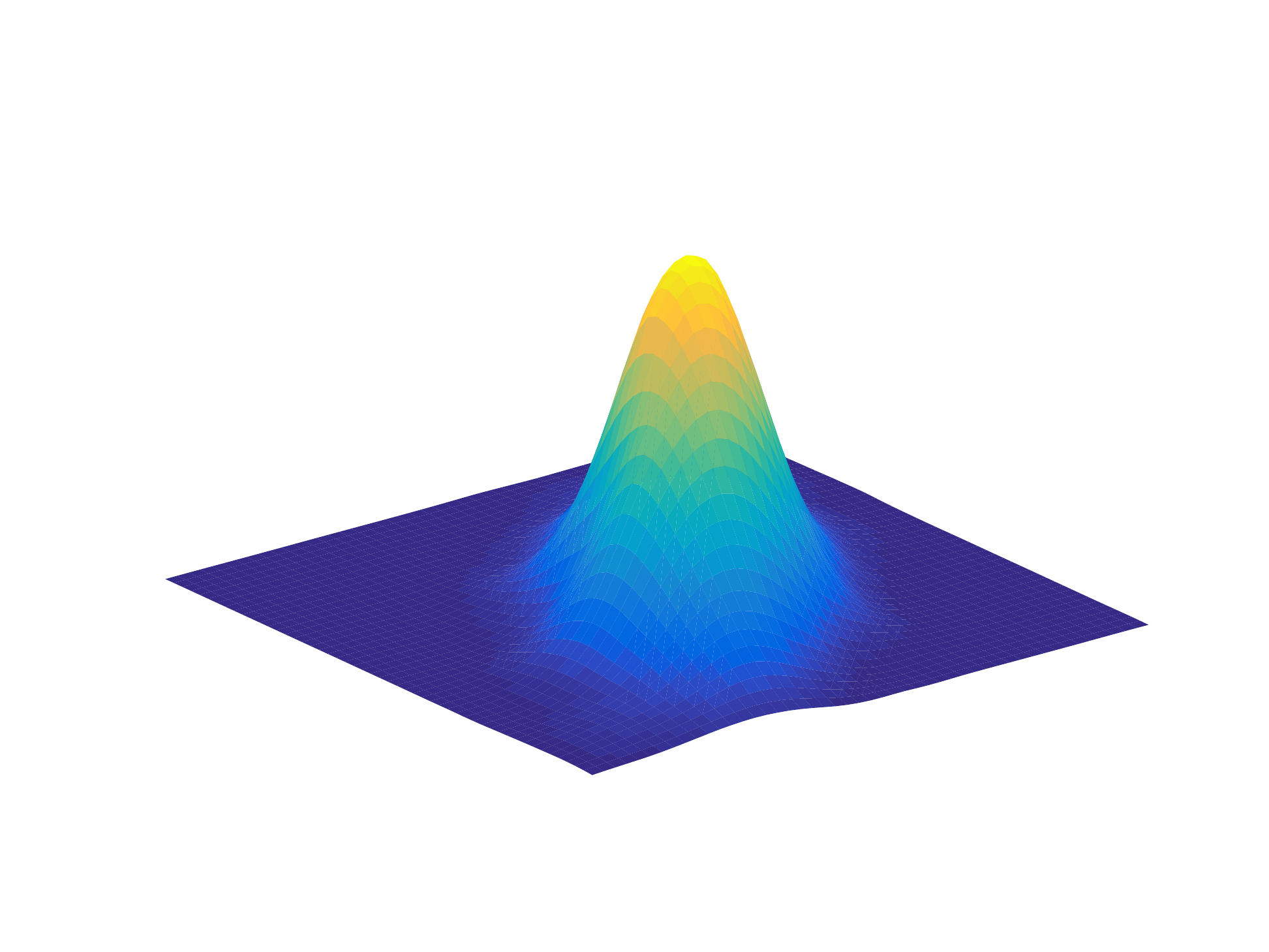}\hspace*{-15pt}
\includegraphics[width=5.5cm]{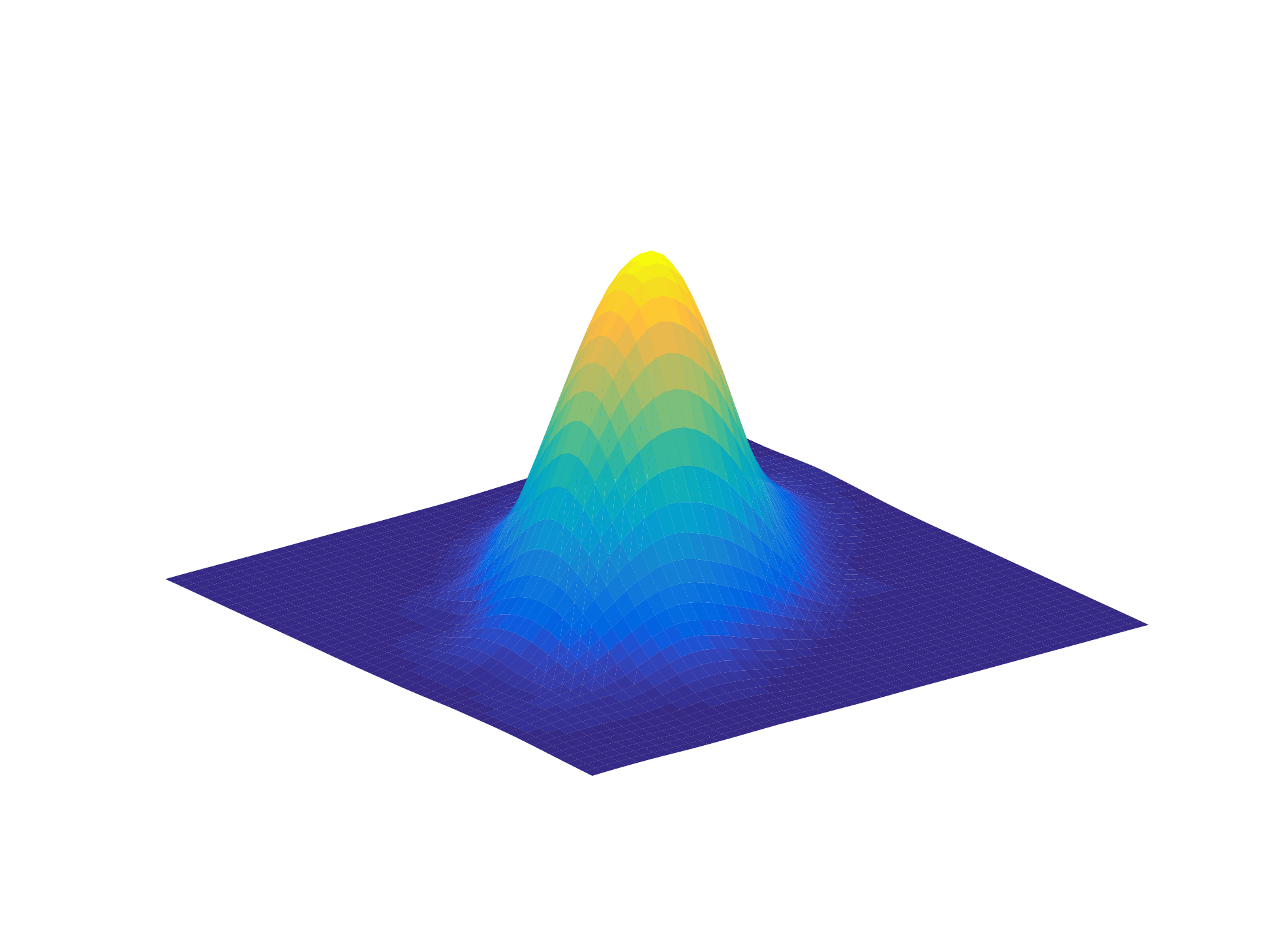}\hspace*{-15pt}
\includegraphics[width=5.5cm]{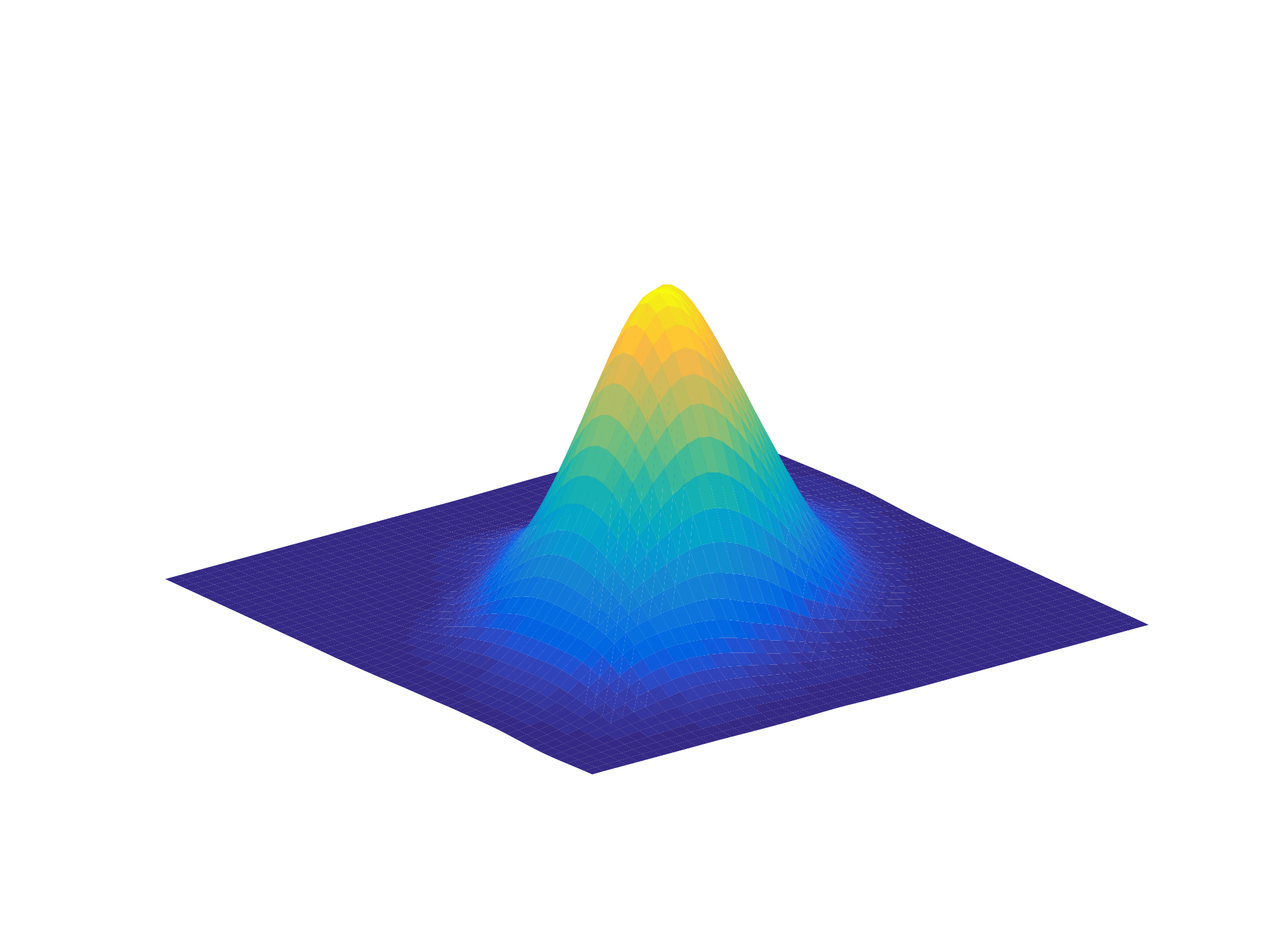}
};
\node[] (p04) at (-5.9,1.1) {\footnotesize $p=0.4$};
\node[] (p05) at (-1,1.1) {\footnotesize $p=0.5$};
\node[] (p06) at (3.9,1.1) {\footnotesize $p=0.6$};

\node[] (p01) at (-5.9,4.75) {\footnotesize $p=0.1$};
\node[] (p02) at (-1,4.75) {\footnotesize $p=0.2$};
\node[] (p03) at (3.9,4.75) {\footnotesize $p=0.3$};

\node[] (p07) at (-5.9,-2.5) {\footnotesize $p=0.7$};
\node[] (p08) at (-1,-2.5) {\footnotesize $p=0.8$};
\node[] (p09) at (3.9,-2.5) {\footnotesize $p=0.9$};
\end{tikzpicture}
\end{center}
\vspace*{-.5cm}
\caption{\emph{Joint distributions of $(S_n,K_n)$ by Monte-Carlo
simulations for $n=10^7$ and varying $p$: the case $p=0.5$
is seen to have stronger dependence than the others.}}
\label{fig-SK}
\end{figure}

These results are to be compared with the corresponding ones for
random $m$-ary search trees \cite{ChFuHwNe}, and the differences for
correlation coefficients are summarized in Table~\ref{tab-tries-mst}.
\begin{table}[!h]
\begin{center}
\begin{tabular}{|c||c|c|}
\hline
trees & $\rho(S_n, K_n)$ & $\rho(S_n, N_n)$ \\ \hline
tries & \multicolumn{1}{c|}{$\left\{\begin{tabular}{@{\ }l@{}}
   $p\neq q:\rightarrow 0$ \\ $p=q:\text{periodic}$
   \end{tabular}\right.$} & $\sim1$ \\ \hline
   \bigcell{c}{$m$-ary \\ search trees} &
   \multicolumn{2}{c|}{$\left\{\begin{tabular}{@{\ }c@{}}
   $3\le m\le26:\rightarrow0$  \\ $m\ge 27:\text{periodic}$
   \end{tabular}\right.$}   \\ \hline
\end{tabular}
\end{center}
\vspace*{-.5cm}
\caption{\emph{A comparison of the correlation coefficients defined
on random tries and on random $m$-ary search trees: the size of
$m$-ary search trees corresponds to the space requirement, and the
KPL and NPL are defined similarly as in tries.}}
\label{tab-tries-mst}
\end{table}
Furthermore, the joint distribution for $m$-ary search trees
undergoes a phase change at $m=26$: if the branching factor $m$
satisfies $3\le m\le 26$, then the space requirement is
asymptotically independent with the KPL and NPL, while for $m\ge27$,
their limiting joint distributions contain periodic fluctuations and
are dependent; see \cite{ChFuHwNe} for more information.

The dependence phenomena as those discovered in this paper are not
limited to random tries and have indeed a wider range of connections.
They also appear in different forms in other structures and
algorithms with an underlying binomial splitting process; see
Flajolet \cite{Fl06} and \cite{FuHwZa,HwFuZa} for references on data
structures, algorithms, conflict resolution protocols and stochastic
models. A typical example is the dependence between the number of
coin-tossings (or bits generated, or bits inspected) and the number
of partitioning rounds in (i) CTM tree algorithm (see Rom and Sidi
\cite{RoSi}), (ii) bucket sort (see \cite{Knuth98} and Mahmoud et al.
\cite{MFJR}), (iii) RS Algorithm for generating random permutations
(see Bacher et al. \cite{BBHT}), and (iv) initializing radio networks
(see Myoupo et al. \cite{MTR03}). We will also present the results
without proof for three other classes of digital trees in the last
section.

Our approach is mostly analytic and it is unknown if our results can
be characterized by probabilistic arguments. Indeed, we believe that
the less expected results we discovered are of special interest to
probabilists as more structural interpretation or characterization
remains to be clarified.

An extended abstract of this paper appeared in the online proceedings of the
27th International Meeting on Probabilistic, Combinatorial and
Asymptotic Methods for the Analysis of Algorithms (Krak\'ow, Poland;
July 4--8, 2016); see \cite{FuHw}. More details of the proofs, as
well as a section on extensions are added in this version (some of
them in an appendix). Also, we corrected the plots of Figure
\ref{fig-SK} in \cite{FuHw}. The extended abstract \cite{FuHw} was
peer-reviewed and we incorporated the comments and suggestions of the
referees into this version.

\section{Covariance and Correlation Coefficient}
\label{cov-and-cor}

In this section, we prove Theorem \ref{main-thm-1} on the asymptotics
of the covariance and correlation coefficient of $S_n$ and $K_n$,
where we content ourselves with a detailed sketch of the method
because similar proofs have been given in \cite{FuHwZa}. In fact,
we will also need the variances of $S_n$ and $K_n$, whose derivations
will be recalled below and which have been known for some time;
see Jacquet and R\'{e}gnier \cite{JaRe},
Kirschenhofer and Prodinger \cite{KiPr}, Kirschenhofer et al.
\cite{KiPrSz}, R\'{e}gnier and Jacquet \cite{ReJa}) and \cite{FuHwZa}. See also Table \ref{tb-SKN} for a brief
summary of these results.

Our method of proof is based on the by-now standard two-stage
approach relying on the theory of analytic de-Poissonization and
Mellin transform whose origin can be traced back to Jacquet and
R\'{e}gnier \cite{JaRe}. See Flajolet et al. \cite{FlGoDu} for a
survey on Mellin transform, and Jacquet and Szpankowski \cite{JaSz}
for a survey on analytic de-Poissonization. For the computation of
the covariance, the manipulation can be largely simplified by the
additional notions of Poissonized variance and admissible functions
further developed in our previous papers \cite{FuHwZa,HwFuZa}.

The starting point of our analysis is the recurrence satisfied by
$S_n$ and $K_n$ in (\ref{dist-rec}). A standard means in the
computation of moments of $S_n$ and $K_n$ is the Poisson generating
function, which corresponds to the moments of $S_n$ and $K_n$ with
$n$ replaced by a Poisson random variable with parameter $z$ (this
step is called \emph{Poissonization}).

More precisely, define the Poisson generating function of
$\mathbb{E}(S_n)$ and that of $\mathbb{E}(K_n)$:
\[
    \tilde{f}_{1,0}(z)
    :=e^{-z}\sum_{n\ge 0}
    \mathbb{E}(S_n)\frac{z^n}{n!},\quad\text{and}\quad
    \tilde{f}_{0,1}(z)
    :=e^{-z}\sum_{n\ge 0}\mathbb{E}(K_n)
    \frac{z^n}{n!}.
\]
Then the recurrences \eqref{dist-rec} lead to the functional equations
\begin{equation}\label{f10-f01}
    \left\{
	\begin{split} \tilde{f}_{1,0}(z)
	&=\tilde{f}_{1,0}(pz)+\tilde{f}_{1,0}(qz)+1-(1+z)e^{-z},\\
	\tilde{f}_{0,1}(z)
    &=\tilde{f}_{0,1}(pz)+\tilde{f}_{0,1}(qz)+z(1-e^{-z}).
	\end{split}\right.
\end{equation}
From these equations, we obtain, by Mellin transform techniques
\cite{FlGoDu},
\begin{equation}\label{exp-means}
	\tilde{f}_{1,0}(z)
	\sim z\mathscr{F}[\cdot](z),
	\qquad\text{and}\qquad
	\tilde{f}_{0,1}(z)
	\sim h^{-1}z\log z+z\mathscr{F}[\cdot](z),
\end{equation}
for large $|z|$ in the half-plane $\Re(z)\ge\varepsilon>0$, where $h$
denotes the entropy of Bernoulli($p$). Then, by Cauchy's integral
representation and analytic de-Poissonization techniques \cite{JaSz},
we obtain precise asymptotic approximations to $\mathbb{E}(S_n)$ and
to $\mathbb{E}(K_n)$; see \cite{FuHwZa} for more details.

Similarly, for the variances $\mathbb{V}(S_n)$ and $\mathbb{V}(K_n)$,
we introduce the Poisson generating functions of the second moments:
\[
    \tilde{f}_{2,0}(z)
    :=e^{-z}\sum_{n\ge0} \mathbb{E}(S_n^2)
    \frac{z^n}{n!}, \quad\text{and}\quad
    \tilde{f}_{0,2}(z)
    :=e^{-z}\sum_{n\ge0}\mathbb{E}(K_n^2)
    \frac{z^n}{n!},
\]
which then satisfy, by \eqref{dist-rec}, the same type of functional
equations as in \eqref{f10-f01} but with different non-homogeneous
parts. Instead of computing directly asymptotic approximations to the
second moments, it proves computational more advantageous to consider
the Poissonized variances
\begin{equation}\label{VS-VK}
    \left\{
	\begin{split}
		\tilde{V}_S(z)
		&:= \tilde{f}_{2,0}(z)-\tilde{f}_{1,0}(z)^2
		-z\tilde{f}_{1,0}'(z)^2,\\
		\tilde{V}_K(z)
		&:= \tilde{f}_{0,2}(z)-\tilde{f}_{0,1}(z)^2
		-z\tilde{f}_{0,1}'(z)^2,
	\end{split}\right.	
\end{equation}
and then following the same Mellin-de-Poissonization approach (as for
the means) to derive the first and the third asymptotic estimate in
the second column of Table~\ref{tb-SKN}; again see \cite{FuHwZa} for
details.

It remains to derive the claimed estimate for the covariance. For
that purpose, we introduce the Poisson generating function
\[
\tilde{f}_{1,1}(z) :=e^{-z}\sum_{n\ge 0}\mathbb{E}(S_nK_n)
\frac{z^n}{n!},
\]
which satisfies, again by \eqref{dist-rec},
\begin{align*}
	\tilde{f}_{1,1}(z)
	&=\tilde{f}_{1,1}(pz)+\tilde{f}_{1,1}(qz)
	+\tilde{f}_{1,0}(pz)\bigl(\tilde{f}_{0,1}(qz)+z\bigr)
	+\tilde{f}_{1,0}(qz)\bigl(\tilde{f}_{0,1}(pz)+z\bigr)\\
%	+z\tilde{f}_{1,0}(pz)+z\tilde{f}_{1,0}(qz)\\
	&\qquad+pz\tilde{f}'_{1,0}(pz)+qz\tilde{f}'_{1,0}(qz)
	+\tilde{f}_{0,1}(pz)
	+\tilde{f}_{0,1}(qz)+z(1-e^{-z}).
\end{align*}
To compute the covariance, it is beneficial to introduce now the
\emph{Poissonized covariance} (see \eqref{VS-VK} or \cite{FuHwZa}
for similar details)
\[
	\tilde{C}(z)=\tilde{f}_{1,1}(z)
	-\tilde{f}_{1,0}(z)\tilde{f}_{0,1}(z)
	-z\tilde{f}'_{1,0}(z)\tilde{f}'_{0,1}(z),
\]
which satisfies
\begin{align}\label{Cz}
	\tilde{C}(z)
	=\tilde{C}(pz)+\tilde{C}(qz)+\tilde{h}_1(z)+\tilde{h}_2(z),
\end{align}
where
\[
	\tilde{h}_1(z)
	=pqz\left(\tilde{f}'_{1,0}(pz)-\tilde{f}'_{1,0}(qz)\right)
	\left(\tilde{f}'_{0,1}(pz)-\tilde{f}'_{0,1}(qz)\right),
\]
and
\begin{equation}\label{h2z}
\begin{split}
	\tilde{h}_2(z)
	&=ze^{-z}\left(\tilde{f}_{1,0}(pz)+\tilde{f}_{1,0}(qz)
	+p(1-z)\tilde{f}'_{1,0}(pz)
	+q(1-z)\tilde{f}'_{1,0}(qz)\right)\\
	&\quad+e^{-z}\left((1+z)\tilde{f}_{0,1}(pz)
	+(1+z)\tilde{f}_{0,1}(qz)
	-pz^2\tilde{f}'_{0,1}(pz)-qz^2\tilde{f}'_{0,1}(qz)
	\right)\\&\quad+ze^{-z}\left(1-(1+z^2)e^{-z}\right).
\end{split}
\end{equation}
Note that $\tilde{h}_1$ is zero when $p=\frac12$. Furthermore, from
(\ref{exp-means}) (which can be differentiated since they hold in a
sector $\mathscr{S} =\{z\in\mathbb{C}\ :\ \Re(z)\geq\epsilon,
\vert{\rm Arg}(z)\vert\le\theta_0\}$ with $0<\theta_0<\pi/2$ in the
complex plane), we obtain that $\tilde{h}_1(z)=O(|z|)$ and
$\tilde{h}_2(z)$ is exponentially small for large $|z|$ in
$\Re(z)>0$. Also $\tilde{h}_1(z)+\tilde{h}_2(z)=O(|z|^2)$ as $z\to0$.
Thus the Mellin transform of $\tilde{h}_1(z)+\tilde{h}_2(z)$ exists
in the strip $\langle -2,-1\rangle$, and we have then the inverse
Mellin integral representation
\begin{align}\label{inv-cz}
	\tilde{C}(z) =
	\frac1{2\pi i}\int_{-\frac32-i\infty}
	^{-\frac32+i\infty}
	\frac{\mathscr{M}[\tilde{h}_1(z)
	+\tilde{h}_2(z);s]}{1-p^{-s}-q^{-s}}\,
	z^{-s}\mathrm{d} s,
\end{align}
where $\mathscr{M}[\phi(z);s] := \int_0^\infty \phi(z)
z^{s-1}\mathrm{d} z$ denotes the Mellin transform of $\phi$; see
\cite{FlGoDu}.

Next, again from (\ref{exp-means}) we see that
$\mathscr{M}[\tilde{h}_1(z);s]$ can be analytically continued to the
vertical line $\Re(s)=-1$ and has no singularities there. Thus, by
shifting the line of integration in (\ref{inv-cz}) and computing
residues, we obtain
\[
    \tilde{C}(z) \sim z\mathscr{F}[g^{(2)}](z),
\]
uniformly for $z$ in a sector.

What is left is the computation of the Fourier coefficients of the
periodic function (see Proposition \ref{fc-cov} below). This is in
fact the most technical part of the proof because $\tilde{h}_1(z)$
contains the product of the two terms $\tilde{f}'_{1,0}(pz)
-\tilde{f}'_{1,0}(qz)$ and $\tilde{f}'_{0,1}(pz)
-\tilde{f}'_{0,1}(qz)$, and thus $\mathscr{M}[\tilde{h}_1(z);s]$
is a Mellin convolution integral. In \cite{FuHwZa}, a general
procedure was given for the simplification of such integrals (see
\cite[p.\ 24 \emph{et seq.}]{FuHwZa}). This simplification procedure
(see Appendix~\ref{sec:app} for details) and a direct application of
the theory of admissible functions of analytic de-Poissonization now
yield the following estimate for the covariance of $S_n$ and $K_n$.
\begin{pro}\label{fc-cov} The covariance of $S_n$ and $K_n$ is
asymptotically linear
\[
    \mathrm{Cov}(S_n,K_n)\sim n\mathscr{F}[g^{(2)}](n).
\]
Here
\begin{align}\label{g2k}
\begin{split}
	g^{(2)}_k
	&=\frac{\Gamma(\chi_k)}{h}\left(1-\frac{\chi_k+2}
	{2^{\chi_k+1}}\right)-\frac{1}{h^2}
	\sum_{j\in\mathbb{Z}\setminus\{0\}}
	\Gamma(\chi_{k-j}+1)(\chi_j-1)\Gamma(\chi_j)\\
	&-\frac{\Gamma(\chi_k+1)}{h^2}
	\left(\gamma+1+\psi(\chi_k+1)
	-\frac{p\log^2p+q\log^2q}{2h}\right)\\
	&+\frac{1}{h}\sum_{\ell\ge 2}
	\frac{(-1)^{\ell}(p^{\ell}+q^{\ell})}
	{\ell!(1-p^{\ell}-q^{\ell})}\,\Gamma(\chi_k+\ell-1)
	(2\ell^2-2\ell+1+\chi_k(2\ell-1)),
\end{split}
\end{align}
where $\gamma$ denotes Euler's constant, $\psi(z)$ is the digamma
function and $\chi_k$ is defined in \eqref{gk}.
\end{pro}

\begin{Rem}
If $\frac{\log p}{\log q}\not\in\mathbb{Q}$, then only $k=0$ is
needed and the second term (the sum over $j$) on the right-hand side
of \eqref{g2k} has to be dropped. Also the first term here
$\frac{\Gamma(\chi_k)}{h} \left(1-\frac{\chi_k+2}
{2^{\chi_k+1}}\right)$ is taken to be its limit $\frac1h(\log
2+\frac12)$ as $\chi_k\to0$ when $k=0$.
\end{Rem}

The asymptotic estimate for the correlation coefficient in Theorem
\ref{main-thm-1} now follows from this and the results for the
variances of $S_n$ and $K_n$ (see Table \ref{tb-SKN}), where
expressions for $g^{(1)}_k$ and $g^{(3)}_k$ can be found, e.g., in
\cite{FuHwZa}. For convenience, we give below the expressions in the
unbiased case. Note that both $\mathscr{F}[g^{(1)}](n)$ and
$\mathscr{F}[g^{(3)}](n)$ are strictly positive; see Schachinger
\cite{Sch} for details.

In the symmetric case, an alternative expression to \eqref{g2k}
(avoiding the convolution of two Fourier series) is
\begin{align*}
	g^{(2)}_k
	&=\frac{\Gamma(\chi_k)\left(1-\frac{\chi_k^2+\chi_k+4}
	{2^{\chi_k+2}}\right)}{\log 2}
	+\frac1{\log 2}\sum_{\ell\ge 1}\frac{(-1)^{\ell}
	\Gamma(\chi_k+\ell)
	\left(\ell(2\ell+1)(\chi_k+\ell)-(\ell+1)^2\right)}
	{(\ell+1)!(2^{\ell}-1)};
\end{align*}
see the discussion of the size of tries in \cite{FuHwZa}, where a
similar alternative expression was given for $g^{(1)}_k$, which reads
\begin{align*}
	g^{(1)}_k
	=-\frac{\Gamma(\chi_k-1)\chi_k(\chi_k+1)^2}{4\log 2}
	+\frac{2}{\log 2}\sum_{\ell\ge 1}\frac{(-1)^{\ell}
	\Gamma(\chi_k+\ell)
	\ell \bigl(\ell(\chi_k+\ell)-1\bigr)}
	{(\ell+1)!(2^{\ell}-1)}.
\end{align*}
Moreover, also in \cite{FuHwZa}, the following expression for
$g^{(3)}_k$ can be found
\begin{align*}
	g^{(3)}_k
	=\frac{\Gamma(\chi_k)\left(1-\frac{\chi_k^2-\chi_k+4}
	{2^{\chi_k+2}}\right)}{\log 2}
	+\frac{2}{\log 2}\sum_{\ell\ge 1}
	\frac{(-1)^{\ell}\Gamma(\chi_k+\ell)(\ell(\chi_k+\ell-1)-1)}
	{\ell!(2^{\ell}-1)}.
\end{align*}
Note that $\chi_k=\frac{2k\pi i}{\log 2}$ and $2^{\chi_k}=1$, and the
reason of retaining $2^{\chi_k+2}$ in the denominator is to give a
uniform expression for all $k$ (notably $k=0$). These provide an
explicit expression for the periodic function $F(n)$ in Theorem
\ref{main-thm-1}. Also, since all the periodic functions have very
small amplitude, the average value of the periodic function $F(z)$ can
be well-approximated by
\[
	\frac{g^{(2)}_0}{\sqrt{g^{(1)}_0g^{(3)}_0}}
	\approx 0.9272416035\cdots.
\]

\section{Limit Law}\label{ll}

In this section, we prove Theorem \ref{main-thm-2}, part (i); the
proof of part (ii) is similar and only sketched. The key tool of the
proof is the multivariate version of the contraction method; see
Neininger and R\"{u}schendorf \cite{NeRu2}. More precisely, we will
use Theorem 3.1 in \cite{NeRu2}.

We first recall the expression for the square-root of a
positive-definite $2\times 2$ matrix
\[
	M=\begin{pmatrix}
	    a & b \\ b & c
	\end{pmatrix}.
\]
It is well-known that such a matrix has exactly one
positive-definite square root which is given by
\begin{align}\label{A-half}
	M^{\frac12}=\frac{1}{\sqrt{a+c+2\sqrt{ac-b^2}}}
	\begin{pmatrix}
		a+\sqrt{ac-b^2} & b \\
		b & c+\sqrt{ac-b^2}
	\end{pmatrix},
\end{align}
with the inverse
\begin{align}\label{A-half-inv}
	M^{-\frac12}=\frac1
	{\sqrt{(ac-b^2)\bigl(a+c+2\sqrt{ac-b^2}\bigr)}}
	\begin{pmatrix}
		c+\sqrt{ac-b^2} & -b \\
		-b & a+\sqrt{ac-b^2}
	\end{pmatrix}.
\end{align}
Now we give the proof of Theorem \ref{main-thm-2}, part (i).

\paragraph{Proof of Theorem \ref{main-thm-2}, Part (i).}
Note first that
\[
	\begin{pmatrix}
		S_n \\
		K_n
	\end{pmatrix}
	\stackrel{d}{=}
	\begin{pmatrix}
	    1 & 0 \\
		0 & 1
	\end{pmatrix}
	\begin{pmatrix} S_{B_n}\\ K_{B_n}\end{pmatrix}
	+\begin{pmatrix} 1 & 0 \\ 0 & 1\end{pmatrix}
	\begin{pmatrix} S_{n-B_n}^{*}\\ K_{n-B_n}^{*}\end{pmatrix}
	+\begin{pmatrix} 1 \\ n\end{pmatrix},
\]
where the notation is as in Section~\ref{intro}. The contraction
method was specially developed for obtaining limiting distribution
results for such recurrences; see \cite{NeRu2}.

\newpage

We need some notation. First, define
\begin{align}\label{Sigma-n}
	\widehat{\Sigma}_n:=\begin{pmatrix}
	    \mathbb{V}(S_n) & {\rm Cov}(S_n,K_n) \\
		{\rm Cov}(S_n,K_n) & \mathbb{V}(K_n)
	\end{pmatrix}.
\end{align}
This matrix is clearly positive-definite for all $n$ sufficiently
large. Next define
\[
	M_n^{(1)}
    :=\widehat{\Sigma}_n^{-\frac12}
    \widehat{\Sigma}_{B_n}^{\frac12},\qquad
	M_n^{(2)}
    :=\widehat{\Sigma}_n^{-\frac12}
    \widehat{\Sigma}_{n-B_n}^{\frac12}
\]
and
\[
	\begin{pmatrix} b_n^{(1)} \\ b_n^{(2)}\end{pmatrix}
	=\widehat{\Sigma}_n^{-\frac12}\begin{pmatrix}
	1-\mu(n)+\mu(B_n)+\mu(n-B_n)\\
	n-\nu(n)+\nu(B_n)+\nu(n-B_n)
	\end{pmatrix},
\]
where $\mu(n)={\mathbb E}(S_n)$ and $\nu(n)={\mathbb E}(K_n)$.

Now to apply the contraction method in \cite{NeRu2}, it suffices
to show that the following conditions hold
\begin{align}
	&b_n^{(i)}\stackrel{L_3}{\longrightarrow} 0,
	\qquad M_n^{(i)}\stackrel{L_3}{\longrightarrow} M_i,
	\label{cond-1}\\
	&\mathbb{E}\bigl(\|M_1\|^3_{\text{op}}
	+\|M_2\|^3_{\text{op}}\bigr)<1,
	\qquad\mathbb{E}\bigl(\|M_n^{(i)}\|^3_{\text{op}}
	\chi_{\{B_n^{(i)}\le j\}\cup\{B_n^{(i)}=n\}}\bigr)
	\longrightarrow 0\label{cond-2}
\end{align}
for $i=1,2$ and $j\in\mathbb{ N}$, where $\stackrel{L_3}
{\longrightarrow}$ denotes convergence in the $L_3$-norm,
$\|\cdot\|_{\text{op}}$ is the operator norm, $\chi_S$ denotes the
characteristic function of set $S$, $B_n^{(1)}=B_n, B_n^{(2)}=n-B_n$
and
\[
	M_1=\begin{pmatrix}
	    \sqrt{p} & 0 \\ 0 & \sqrt{p}
	\end{pmatrix},\qquad
	M_2=\begin{pmatrix}
	    \sqrt{q} & 0 \\ 0 & \sqrt{q}
	\end{pmatrix}.
\]
Then the contraction method in \cite{NeRu2} guarantees that
$(S_n,K_n)$ (centralized and normalized) converges in distribution to
the unique fixed-point with mean $0$, covariance matrix the unity
matrix and finite $L_3$-norm of
\[
	\begin{pmatrix}X_1\\ X_2\end{pmatrix}
	\stackrel{d}{=}
	\begin{pmatrix}\sqrt{p} & 0 \\ 0 & \sqrt{p}\end{pmatrix}
	\begin{pmatrix}X_1\\ X_2\end{pmatrix}
	+\begin{pmatrix}\sqrt{q} & 0 \\ 0 & \sqrt{q}\end{pmatrix}
	\begin{pmatrix}X_1^{*}\\ X_2^{*}\end{pmatrix},
\]
where $(X_1^{*},X_2^{*})$ is an independent copy of $(X_1,X_2)$.
Obviously, the bivariate normal distribution is the solution.
All this is summarized as follows.

\begin{pro}\label{asym-clt} The following convergence in distribution
holds:
\[
	\widehat{\Sigma}_n^{-\frac12}
	\begin{pmatrix}
		S_n-\mathbb{E}(S_n)\\
		K_n-\mathbb{E}(K_n)
	\end{pmatrix}
		\stackrel{d}{\longrightarrow} \mathcal{N}_2(0, I_2).
\]
\end{pro}
\begin{proof}
We only check \eqref{cond-1} because the second condition of (\ref{cond-2}) follows along similar lines and the first condition of (\ref{cond-2}) follows from \eqref{cond-1} in view of
\[
	\|M_1\|_{\text{op}}=\sqrt{p}\qquad\text{and}
	\qquad \|M_2\|_{\text{op}}=\sqrt{q}.
\]

We start with proving \eqref{cond-1} for $b_n^{(i)}$ for which
we use the notations
\[
	\Omega_1(n)=\mathbb{V}(S_n),
	\quad\Omega_2(n)={\rm Cov}(S_n,K_n),
	\quad\Omega_3(n)=\mathbb{V}(K_n)
\]
and
\[
    D(n)=\Omega_1(n)\Omega_3(n)-\Omega_2(n)^2.
\]
Also define
\[
    R(n) = \Omega_1(n)+\Omega_3(n)+2\sqrt{D(n)}.
\]

Then, by \eqref{A-half}, we see that
\begin{align*}
	b_n^{(1)}&=(1-\mu(n)+\mu(B_n)
	+\mu(n-B_n))\frac{\Omega_3(n)
	+\sqrt{D(n)}}{\sqrt{D(n)R(n)}}\\
	&\quad-(n-\nu(n)+\nu(B_n)
	+\nu(n-B_n))\frac{\Omega_2(n)}{\sqrt{D(n)R(n)}}
\end{align*}
and a similar expression for $b_n^{(2)}$ holds. From the normality of
both $S_n$ and $K_n$ (proved for $S_n$ via the contraction method in
\cite{FuLe2} and a similar method of proof also applies to $K_n$), we
have
\[
	\frac{1-\mu(n)+\mu(B_n)+\mu(n-B_n)}{\sqrt{n}}
	\stackrel{L_3}{\longrightarrow} 0
\]
and
\[
	\frac{n-\nu(n)+\nu(B_n)+\nu(n-B_n)}{\sqrt{n\log n}}
	\stackrel{L_3}{\longrightarrow} 0.
\]
Moreover, we have
\[
	\sqrt{n}\,\frac{\Omega_3(n)+\sqrt{D(n)}}{\sqrt{D(n)R(n)}}
	\sim\frac{1}
	{\sqrt{\mathscr{F}[g^{(1)}](n)}},
\]
and
\[
	\sqrt{n\log n}\,\frac{\Omega_2(n)}
	{\sqrt{D(n)R(n)}}
	\sim\frac{\mathscr{F}[g^{(2)}](n)}
	{\lambda\sqrt{\log n\mathscr{F}[g^{(1)}](n)}},
\]
where $g^{(1)}, g^{(2)}$ and $\lambda$ are as above. Thus, both
sequences are bounded and, consequently, we obtain the claimed result
with $L_3$-convergence above. Similarly, one proves \eqref{cond-1}
for $b_n^{(2)}$.

Next, we consider $M_n^{(i)}$. Here, we only show the claim for the
$(1,1)$ entry of $M_n^{(1)}$ (denoted by $M_n^{(1)}(1,1)$) all other
cases being treated similarly. First, observe that by definition and
matrix square-root, we have
\begin{align*}
	M_n^{(1)}(1,1)
	=\frac{\sqrt{R(n)}}{\sqrt{R(B_n)}}
	\cdot\frac{(\Omega_3(n)
	+\sqrt{D(n)})(\Omega_1(B_n)+\sqrt{D(B_n)})
	-\Omega_2(n)\Omega_2(B_n)}{\sqrt{D(n)R(n)}}.
\end{align*}
Now, from the strong law of large numbers for the binomial
distribution
\[
    \frac{B_n}{n}\stackrel{\text{a.s.}}{\longrightarrow} p
\]
and from Taylor series expansion (note that all periodic functions are
infinitely differentiable), we have
\[
	\frac{\sqrt{R(n)}}{\sqrt{R(B_n)}}
	\stackrel{\text{a.s.}}{\longrightarrow}\frac{1}{\sqrt{p}},
\]
and
\[
	\frac{(\Omega_3(n)
	+\sqrt{D(n)})(\Omega_1(B_n)+\sqrt{D(B_n)})
	-\Omega_2(n)\Omega_2(B_n)}{\sqrt{D(n)R(n)}}
	\stackrel{\text{a.s.}}{\longrightarrow}p.
\]
Thus, $M_n^{(1)}(1,1)\stackrel{\text{a.s.}} {\longrightarrow}
\sqrt{p}$ from which the claim follows by the dominated convergence
theorem.
\end{proof}

Next, set
\[
    \widetilde{\Sigma}_n
	:=\begin{pmatrix}
	    n\mathscr{F}[g^{(1)}](n) & 0 \\
		0 & \lambda n\log n
	\end{pmatrix}.
\]
Then, we have the following simple lemma.
\begin{lmm}
We have, as $n\rightarrow\infty$,
\[
    \widehat{\Sigma}_n^{-\frac12}\widetilde{\Sigma}_n^{\frac12}
	\rightarrow I_2.
\]
\end{lmm}
\pf This follows by a straightforward computation using the
expressions for the matrix square-root \eqref{A-half} and its inverse
\eqref{A-half-inv}. For example, the entry $(1,2)$ of
$\widehat{\Sigma}_n^{-\frac12} \widetilde{\Sigma}_n^{\frac12}$ (where
we use the notations from the proof of the previous proposition)
satisfies
\[
	-\frac{\Omega_2(n)\sqrt{\lambda n\log n}}
	{\sqrt{D(n)R(n)}}
	\sim-\frac{\mathscr{F}[g^{(2)}](n)}
	{\sqrt{\lambda\log n\mathscr{F}[g^{(1)}](n)}},
\]
which tends to $0$ as claimed. The other entries are treated
similarly, \qed

Theorem \ref{main-thm-2}, part (i) now follows from this lemma and
Proposition \ref{asym-clt}.

Next, we sketch the (similar) proof of Theorem \ref{main-thm-2}, part
(ii).

\paragraph{Proof of Theorem \ref{main-thm-2}, Part (ii).} The proof
runs along similar lines as in Part (i). The only difference is that
now it is not entirely obvious that $\widehat{\Sigma}_n$ is positive
definite. Note, however, that from the discussion in the
introduction, this matrix is positive-definite if and only if
$\Sigma_n$ (defined in Theorem \ref{main-thm-2}) is positive
definite. This is ensured by the following lemma.

\begin{lmm}
$\Sigma_n$ is positive-definite for all $n$ large enough.
\end{lmm}
\pf It suffices to show that $\det(\Sigma_n)>0$ for all $n$ large
enough. Indeed, we have
\begin{align*}
    \det(\Sigma_n)&=n^2\mathscr{F}[g^{(1)}](n)
    \mathscr{F}[g^{(3)}](n)-n^2(\mathscr{F}[g^{(1)}](n))^2\\
    &=n^2\mathscr{F}[g^{(1)}](n)
    \mathscr{F}[g^{(3)}](n)(1-F(\log_2 n)^2),
\end{align*}
from which the result follows.

Note that this in addition shows the stronger result
$\det(\Sigma_n)\ge dn^2$ for all $n$ large enough where $d>0$. (A
proof avoiding numerical computations can be performed using the same
approach as in Proposition 3 of \cite{FuLe}.) \qed

The rest of the proof is similar as in the asymmetric case and is
omitted.

\section{Extensions}

In this section, we show that the dependence phenomena we discovered
here on random binary tries (Theorem~\ref{main-thm-1} and
Theorem~\ref{main-thm-2}) also find their appearance in other trees
and structures whose subtree-sizes and sub-structure-sizes are
dictated by a binomial or a multinomial distribution.

\newpage

For simplicity, we consider in this section only three varieties of
random digital trees: random $m$-ary tries, random PATRICIA tries and
random bucket digital search trees; see \cite{FuHwZa} for more
potential examples with the same splitting principles.

\paragraph{$m$-ary Tries.} It is straightforward to extend our tries
constructed from binary input strings to inputs from an $m$-ary
alphabets, $m\ge2$. In this case, the resulting trie becomes an
$m$-ary tree (since each node now has $m$ subtrees one belonging to
each letter). As a random model, we assume that bits are generated
independently at random with the $i$-th letter occurring with
probability $p_i$, where $p_1+\cdots+p_m=1$ and $0<p_i<1$ for $1\le
i\le m$.

The size and the key path length (which we again denote by $S_n$ and
$K_n$) in such random $m$-ary tries satisfy the recurrences
\[
    \left\{
	\begin{split}
		S_n&\stackrel{d}{=}S_{I_n^{(1)}}^{(1)}
        +\cdots+S_{I_n^{(m)}}^{(m)}+1,\\
		K_n&\stackrel{d}{=}K_{I_n^{(1)}}^{(1)}
        +\cdots+K_{I_n^{(m)}}^{(m)}+n,\\
	\end{split}\qquad(n\ge2),\right.
\]
with the initial conditions $S_n=K_n=0$ for $n\le1$, where
$(S_{n}^{(i)})$ and $(K_n^{(i)})$ are independent copies of $(S_n)$
and $(K_n)$, respectively, for $1\le i\le m$, and
\begin{align}\label{I}
    {\mathbb P}(I_n^{(1)}=j_1,\ldots,I_n^{(m)}=j_m)
    =\binom{n}{j_1,\ldots,j_m}p_1^{j_1}\cdots p_m^{j_m},
\end{align}
for all $j_1,\ldots,j_m\ge 0$ with $j_1+\cdots+j_m=n$.

The pair $(S_n,K_n)$ satisfies the same type of properties as those
described in Theorem~\ref{main-thm-1} and Theorem~\ref{main-thm-2}
for binary tries, where the symmetric case here corresponds to
$p_1=\cdots=p_m=1/m$ and all other cases are asymmetric. Only the
expressions for $g^{(1)}_k,g^{(2)}_k,g^{(3)}_k$ and $\lambda$ are different
but they can be computed via the same analytic tools as those used in
\cite{FuHwZa}. For the sake of simplicity, we only give the
expressions in the symmetric case ($\chi_k=2k\pi i/\log m$) as
follows:
\begin{align*}
	g^{(1)}_k
	&=\frac{\Gamma(\chi_k-1)
	\bigl(\chi_k-\frac{\chi_k^3+2\chi_k^2+5\chi_k}
	{2^{\chi_k+2}}\bigr)}{\log m}
	+\frac{2}{\log m}\sum_{\ell\ge 1}\frac{(-1)^{\ell}
	\Gamma(\chi_k+\ell)
	\ell \bigl(\ell(\chi_k+\ell)-1\bigr)}
	{(\ell+1)!(m^{\ell}-1)}, \\
% \end{align*}
% and
% \begin{align*}
	g^{(2)}_k
	&=\frac{\Gamma(\chi_k)\bigl(1-\frac{\chi_k^2+\chi_k+4}
	{2^{\chi_k+2}}\bigr)}{\log m}
	+\frac1{\log m}\sum_{\ell\ge 1}\frac{(-1)^{\ell}
	\Gamma(\chi_k+\ell)
	\left(\ell(2\ell+1)(\chi_k+\ell)-(\ell+1)^2\right)}
	{(\ell+1)!(m^{\ell}-1)}, \\
% \end{align*}
% and
% \begin{align*}
	g^{(3)}_k
	&=\frac{\Gamma(\chi_k)\bigl(1-\frac{\chi_k^2-\chi_k+4}
	{2^{\chi_k+2}}\bigr)}{\log m}
	+\frac{2}{\log m}\sum_{\ell\ge 1}
	\frac{(-1)^{\ell}\Gamma(\chi_k+\ell)(\ell(\chi_k+\ell-1)-1)}
	{\ell!(m^{\ell}-1)}.
\end{align*}
Note that the variance of the size was considered in \cite{FuLe}, but
no explicit expression was given for the Fourier coefficients of the
periodic function.

With the help of these expressions, we obtain the following
numerical approximations to the average value of the periodic
function of the correlation coefficient between $S_n$ and $K_n$ in
the symmetric case. We see that they differ little.
\begin{table}[!h]
\begin{center}
\begin{tabular}{|c||c|c|c|c|c|} \hline
	$m$ & $2$ & $3$ & $4$ & $5$ & $6$\\ \hline
	$\begin{array}{c}
	\text{average value of the}\\
	\text{periodic fluctuations}
	\end{array}$ & $0.927$
	& $0.925$ & $0.924$ & $0.922$ & $0.921$
	\\ \hline
\end{tabular}
\end{center}
\vspace*{-.4cm}
\caption{\emph{Numerical approximations to the average values of
the periodic functions $F(n)$ arising in the asymptotic estimate
$\rho(S_n,K_n)\sim F(n)$ for the symmetric case and $m=2,\dots,6$.}}
\end{table}

\paragraph{PATRICIA Tries.} A simple idea to increase the efficiency
of tries is to remove all internal nodes with one-way branching. The
resulting tree is called a PATRICIA trie; here PATRICIA is an acronym
of ``Practical Algorithm To Retrieve Information Coded In
Alphanumeric''.

\newpage

We use the same random model as we used above for $m$-ary tries and
consider the size and key-path length of PATRICIA tries (which we
again denote by $S_n$ and $K_n$). Then they satisfy the recurrences
\[
	\begin{split}
		S_n&\stackrel{d}{=}S_{I_n^{(1)}}^{(1)}
        +\cdots+S_{I_n^{(m)}}^{(m)}+T_n,\\
		K_n&\stackrel{d}{=}K_{I_n^{(1)}}^{(1)}
        +\cdots+K_{I_n^{(m)}}^{(m)}+nT_n,\\
	\end{split}\qquad(n\ge2),
\]
with the initial conditions $S_n=K_n=0$ for $n\ge1$, where
$I_n^{(i)}$ is defined as in \eqref{I} above, $(S_n^{(i)})$ and
$(K_n^{(i)})$ are independent copies of $(S_n)$ and $(K_n)$,
respectively, and
\[
    T_n=\begin{cases}
	    1,&\text{if}\ I_n^{(i)}<n\
	        \text{for all}\ 1\le i\le m;\\
	    0,&\text{otherwise}.
    \end{cases}
\]

Note that for $m=2$, the size is deterministic. We thus assume $m\ge
3$ to avoid trivialities. Then the dependence of $(S_n,K_n)$
satisfies \emph{mutatis mutandis} Theorem~\ref{main-thm-1} and
Theorem~\ref{main-thm-2}. In particular, the required changes for
$g^{(1)}_k,g^{(2)}_k,g^{(3)}_k$ in the symmetric case are given as follows
($\chi_k=2k\pi i/\log m$):
\begin{align*} \small
    g^{(1)}_k&=\frac{(m-1)\Gamma(\chi_k-1)}
    {\log m}\Bigg(-1-\frac{(m-1)(\chi_k+1)}{2^{\chi_k}}\\
    &\quad+\biggl(1-\frac{1}{m}\biggr)^{-\chi_k}
    \biggl(1-\frac{(m-1)\chi_k+m+1}
    {2^{\chi_k}}\biggr)+\left(2
    -\frac{1}{m}\right)^{-\chi_k}(2(m-1)\chi_k+2m)\Bigg)\\
    &\quad+\frac{2(m-1)^2}{\log m}
    \sum_{\ell\ge 1}\frac{(-1)^{\ell+1}\Gamma(\ell+\chi_k)
    \ell\bigl(1-\left(1-\frac{1}{m}\right)^{\ell}\bigr)
    \bigl(1-\left(1-\frac{1}{m}\right)^{-\ell-\chi_k}\bigr)}
    {(\ell+1)!(m^{\ell}-1)}, \\
% \end{align*}
% and
% \begin{align*}
	g^{(2)}_k&=\frac{\Gamma(\chi_k)}
	{\log m}\left(\left(1-\frac{1}{m}\right)^{-\chi_k}
	\left(1-\frac{(m-1)\chi_k+2}{2^{\chi_k+1}}\right)+
	\left(2-\frac{1}{m}\right)^{-\chi_k}
	\frac{(m-1)^2\chi_k}{2m-1}\right)\\
	&\quad+\frac{m-1}{\log m}
	\sum_{\ell\ge 1}\frac{(-1)^{\ell+1}\Gamma(\ell+\chi_k)\ell}
	{(\ell+1)!(m^{\ell}-1)}\Bigg((\ell+1)
	\left(1-\frac{1}{m}\right)^{\ell}\\
	&\quad+(\chi_k-1)\left(1-\frac{1}{m}\right)^{-\chi_k}
	-(\ell+\chi_k)
	\left(1-\frac{1}{m}\right)^{-\ell-\chi_k}\Bigg),\\
% \end{align*}
% and
% \begin{align*}
	g^{(3)}_k&=\frac{\Gamma(\chi_k)}
	{\log m}\left(1-\frac{1}{m}\right)^{-\chi_k}
	\left(1+\frac{\chi_k}{m-1}-
	\frac{(m-1)\chi_k^2-(m-3)\chi_k+4(m-1)}
	{(m-1)2^{\chi_k+2}}\right)\\
	&\quad+\frac{2(m-1)^{-\chi_k}}{\log m}
	\sum_{\ell\ge 1}\frac{(-1)^{\ell}
	\Gamma(\ell+\chi_k+1)}{(\ell-1)!(m^{\ell}-1)}.
\end{align*}

These expressions are also valid for $m=2$, where $g^{(1)}_k$ and
$g^{(2)}_k$ can be shown to be identically zero. Note that the result
for the variance of the key-path length was already derived in
\cite{FuHwZa} (for $m=2$) and that for the size was established in
\cite{FuLe} but without a precise expression for the Fourier
coefficients.

Again, we can use the above expressions to obtain the average value
of the periodic function of the correlation coefficient between $S_n$
and $K_n$ in the symmetric case. Note that unlike tries, these values
increase with $m$.
\begin{table}[!h]
\begin{center}
\begin{tabular}{|c||c|c|c|c|} \hline
	$m$ & $3$ & $4$ & $5$ & $6$\\ \hline
	$\begin{array}{c}
	\text{average value of the}\\
	\text{periodic fluctuations}
	\end{array}$ & $0.751$
	& $0.814$ & $0.841$ & $0.856$
	\\ \hline
\end{tabular}
\end{center}
\vspace*{-.4cm}
\caption{\emph{Numerical approximations to the average values of
the periodic functions $F(n)$ arising in the asymptotic estimate
$\rho(S_n,K_n)\sim F(n)$ for the symmetric case and $m=3,\dots,6$.}}
\end{table}

\paragraph{Bucket Digital Search Trees.} Digital search trees (DST)
represent yet another class of digital tree structures; see
\cite{Knuth98,Mahmoud92} for more information. In contrast to tries
and PATRICIA tries, they only have one type of nodes where data are
stored. More precisely, given a set of data consisting of $n$
infinite $0$-$1$ strings, a DST is constructed as follows: if $n=1$,
then the DST consists of only one node holding the sole string;
otherwise, the first string is stored in the root and all others are
directed to the subtrees according to their first bit being $0$ or
$1$; then, the subtrees are built recursively but by using
consecutive bits to split the data.

Clearly, the size of such a DST
is deterministic and equals the input cardinality. We consider
instead a bucket version with an additional capacity $b\ge2$,
allowing each node holding up to $b$ strings and nodes having
subtrees only when they are filled up.

We adopt the same Bernoulli random model as for random tries and
consider the size and key-path length in random bucket digital search
trees (again denoted by $S_n$ and $K_n$), which then satisfy
\[
	\left\{
	\begin{split}
		S_{n+b}&\stackrel{d}{=}S_{B_n}+S_{n-B_n}^{*}+1,\\
		K_{n+b}&\stackrel{d}{=}K_{B_n}+K_{n-B_n}^{*}+n,
	\end{split}\right. \qquad(n\ge 0),
\]
with the initial conditions $S_0=K_0=K_1=\cdots=K_{b-1}=0$ and
$S_1=\cdots=S_{b-1}=1$, where $(S_n^*)$ and $(K_n^*)$ are independent
copies of $(S_n)$ and $(K_n)$, respectively.

The same dependence phenomena as those described in
Theorem~\ref{main-thm-1} and Theorem~\ref{main-thm-2} also hold for
the pair $(S_n,K_n)$. The computation of the sequences
$g^{(1)}_k,g^{(2)}_k,g^{(3)}_k$ is nevertheless more intricate. In the
asymmetric case, one can again use analytic de-Poissonization and
Mellin transform techniques, however, the resulting expressions are
less explicit. On the other hand, in the symmetric case, explicit
expressions for $g^{(1)}_k,g^{(2)}_k,g^{(3)}_k$ are available via the
Poisson-Laplace-Mellin method from \cite{HwFuZa}. As the expressions
are long, we omit them here. Note that the results for the variances
of $S_n$ and $K_n$ have already been obtained in \cite{HwFuZa}.

\paragraph{Other Shape Parameters.} Theorem~\ref{main-thm-1} and
Theorem~\ref{main-thm-2} also extend to pairs of random variables
where the size is replaced by the number of various patterns (such as
the number of internal-external nodes discussed, e.g., by Flajolet
and Sedgewick in \cite{FlSe}) and the key-path length is replaced by
other notions of the path length (such as the total path length of
internal-external nodes).

\section*{Acknowledgments}

The first author acknowledges partial supported by MOST under the
grants MOST-104-2923-M-009-006-MY3 and MOST-105-2115-M-009-010-MY2.
We also thank the helpful comments by the referees for the extended
abstract \cite{FuHw} of this paper.

\appendix
\addcontentsline{toc}{section}{Appendices}

\section{A Sketch of Proof of \eqref{g2k}}\label{sec:app}

We sketch here some details of how the expression \eqref{g2k} for
$g^{(2)}_k$ in Proposition \ref{fc-cov} is obtained. The method we
use is based on that introduced in \cite{FuHwZa}; see p. 25 \emph{et
seq.}

\newpage

First, by moving the line of integration in (\ref{inv-cz}) to the
right and using the residue theorem, we have
\[
    g^{(2)}_k=\frac{G_2(-1+\chi_k)}{h},
\]
where $G_2(s)={\mathscr M}[\tilde{h}_1(z)+\tilde{h}_2(z);s]$. Note
that in the above expression and in what follows, if $\log p/\log q$
is irrational, then only the term with $k=0$ is retained. Thus, our
problem boils down to the computation of $G_2(s)$.

We first consider the Mellin transform of $\tilde{h}_2(z)$ which is
easier to handle. By the expression \eqref{h2z} for $\tilde{h}_2(z)$
from Section \ref{cov-and-cor}, the Mellin transform is given by
\[
    {\mathscr M}[\tilde{h}_2(z);s]
    =\Gamma(s+1)\left(1-\frac{s^2+3s+6}{2^{s+3}}\right)+Y(s),
\]
where
\begin{align*}
    Y(s)=\int_{0}^{\infty}&e^{-z}
    \Big(z\tilde{f}_{1,0}(pz)+pz(1-z)
    \tilde{f}'_{1,0}(pz)+z\tilde{f}_{1,0}(qz)
    +qz(1-z)\tilde{f}'_{1,0}(qz)\\
    &+(1+z)\tilde{f}_{0,1}(pz)
    -pz^2\tilde{f}'_{0,1}(pz)+(1+z)
    \tilde{f}_{0,1}(qz)-qz^2\tilde{f}'_{0,1}(qz)
    \Big)z^{s-1}{\rm d}z.
\end{align*}
Observe that by applying the Mellin transform and its inverse
to (\ref{f10-f01}), we obtain
\begin{equation}\label{expf10}
    \tilde{f}_{1,0}(z)
    =\frac{1}{2\pi i}\int_{(-3/2)}
    \frac{-(\omega+1)\Gamma(\omega)}{1-p^{-\omega}
    -q^{-\omega}}\,z^{-\omega}{\rm d}\omega
\end{equation}
and
\begin{equation}\label{expf01}
    \tilde{f}_{0,1}(z)
    =\frac{1}{2\pi i}\int_{(-3/2)}
    \frac{-\omega\Gamma(\omega)}
    {1-p^{-\omega}-q^{-\omega}}\,z^{-\omega}{\rm d}\omega.
\end{equation}
Substituting these into the integral representation of $Y(s)$ and
interchanging the integrals, we see that
\begin{align*}
    Y(s)&=\frac{1}{2\pi i}\int_{(-3/2)}\frac{\Gamma(\omega)
    \Gamma(s-\omega)}{1-p^{-\omega}-q^{-\omega}}
    \left(p^{-\omega}+q^{-\omega}\right)
    \left(\omega^2-(s-\omega)(2\omega^2+3\omega+1)\right)
    {\rm d}\omega\\
    &=\sum_{\ell\ge 2}\frac{(-1)^{\ell}
    (p^{\ell}+q^{\ell})}
    {\ell!(1-p^{\ell}-q^{\ell})}\,\Gamma(s+\ell)
    \left(\ell^2-(s+\ell)(2\ell^2-3\ell+1)\right),
\end{align*}
where the last line follows from moving the vertical line of
integration to minus infinity and summing over all the residues of
the poles encountered.

For ${\mathscr M}[\tilde{h}_1(z);-1+\chi_k]$, we use the expression
for $\tilde{h}_1(z)$ in Section \ref{cov-and-cor},
(\ref{expf10}) and (\ref{expf01}), and Mellin convolution, giving
\[
    \frac{1}{2\pi i}\int_{(0)+}\frac{pq\left(p^{-\omega}
    -q^{-\omega}\right)\left(p^{\omega}-q^{\omega}\right)}
    {(1-p^{1-\omega}-q^{1-\omega})
    (1-p^{1+\omega}-q^{1+\omega})}\,
    \Gamma(\omega+1)(\chi_k-\omega-1)
    \Gamma(\chi_k-\omega){\rm d}\omega,
\]
where the integration path is the imaginary axis with a small
indentation to the right at the zeros of
$1-p^{1-\omega}-q^{1-\omega}$. Now by the decomposition
\[
    \frac{pq\left(p^{-\omega}-q^{-\omega}\right)
    \left(p^{\omega}-q^{\omega}\right)}
    {(1-p^{1-\omega}-q^{1-\omega})
    (1-p^{1+\omega}-q^{1+\omega})}=
    \frac{1}{1-p^{1-\omega}-q^{1-\omega}}
    +\frac{p^{1+\omega}+q^{1+\omega}}
    {1-p^{1+\omega}-q^{1+\omega}},
\]
the above integral is rewritten as
\begin{equation}\label{asym-case}
    \frac{1}{2\pi i}\int_{(0)+}\left(\frac{1}
    {1-p^{1-\omega}-q^{1-\omega}}
    +\frac{p^{1+\omega}+q^{1+\omega}}
    {1-p^{1+\omega}-q^{1+\omega}}
    \right)\Gamma(\omega+1)(\chi_k-\omega-1)
    \Gamma(\chi_k-\omega){\rm d}\omega.
\end{equation}
We break now this integral into two parts according to the two terms
in the bracket. For the first part, we use the substitution
$\omega\leftrightarrow\chi_k-\omega$ and standard residue calculus,
and obtain
\begin{align*}
    \frac{1}{2\pi i}
    \int_{(0)-}&\frac{1}{1-p^{1+\omega}-q^{1+\omega}}\,
    \Gamma(\chi_k-\omega+1)(\omega-1)
    \Gamma(\omega){\rm d}\omega\\
    =&-\frac{\Gamma(\chi_k+1)}{h}
    \left(\gamma+1+\psi(\chi_k+1)-
    \frac{p\log^2p+q\log^2q}{2h}\right)\\
    &-\frac{1}{h}\sum_{j\in{\mathbb Z}\setminus\{0\}}
    \Gamma(\chi_{k-j}+1)(\chi_j-1)\Gamma(\chi_j)\\
    &+\frac{1}{2\pi i}\int_{(0)+}\frac{1}
    {1-p^{1+\omega}-q^{1+\omega}}\,
    \Gamma(\chi_k-\omega+1)(\omega-1)
    \Gamma(\omega){\rm d}\omega,
\end{align*}
where the second line follows by moving the line of integration over
the imaginary axis and $\psi(s)$ denotes the derivative of
$\log\Gamma(s)$. Next, note that
\begin{align*}
    \frac{1}{2\pi i}\int_{(0)+}&\frac{1}
    {1-p^{1+\omega}-q^{1+\omega}}\,
    \Gamma(\chi_k-\omega+1)(\omega-1)
    \Gamma(\omega){\rm d}\omega \\
% \end{align*}
% where the second line follows by moving the line of integration over
% the imaginary axis and collecting the residues. To make the last
% integral similar to the one from the second part, we rewrite it as
% \begin{align*}
    % \frac{1}{2\pi i}\int_{(0)+}
    % &\frac{1}{1-p^{1+\omega}-q^{1+\omega}}\Gamma(\chi_k-\omega+1)
    % (\omega-1)\Gamma(\omega){\rm d}\omega\\
    &=\frac{1}{2\pi i}\int_{(0)+}\Gamma(\chi_k-\omega+1)
    (\omega-1)\Gamma(\omega){\rm d}\omega\\
    &\qquad+\frac{1}{2\pi i}\int_{(0)+}
    \frac{p^{1+\omega}+q^{1+\omega}}
    {1-p^{1+\omega}-q^{1+\omega}}
    \Gamma(\chi_k-\omega+1)(\omega-1)
    \Gamma(\omega){\rm d}\omega.
\end{align*}
The first integral on the right-hand side is a Mellin convolution
integral and can be evaluated explicitly as
\begin{align*}
    \frac{1}{2\pi i}\int_{(0)+}
    \Gamma(\chi_k-\omega+1)(\omega-1)
    \Gamma(\omega){\rm d}\omega
    &=\int_{0}^{\infty}e^{-z}z^{\chi_k}
    (1-(1-z)e^{-z}){\rm d}z
    -\Gamma(\chi_k+1)\\ &=\Gamma(\chi_k+1)
    \frac{\chi_k-1}{2^{\chi_k+2}}.
\end{align*}
For the second integral, we move the line of integration to infinity
and use the residue theorem, yielding
\begin{align*}
    \frac{1}{2\pi i}\int_{(0)+}
    &\frac{p^{1+\omega}+q^{1+\omega}}
    {1-p^{1+\omega}-q^{1+\omega}}\,
    \Gamma(\chi_k-\omega+1)(\omega-1)
    \Gamma(\omega){\rm d}\omega\\
    &=\sum_{\ell\ge 2}\frac{(-1)^{\ell}(p^{\ell}+q^{\ell})}
    {(\ell-1)!(1-p^{\ell}-q^{\ell})}\,
    \Gamma(\chi_k+\ell-1)(\ell-1)(\chi_k+\ell-2).
\end{align*}
In a similar way, the second part of (\ref{asym-case}) has the series
representation
\begin{align*}
    \frac{1}{2\pi i}\int_{(0)+}&\frac{p^{1+\omega}+q^{1+\omega}}
    {1-p^{1+\omega}-q^{1+\omega}}\,\Gamma(\omega+1)(\chi_k-\omega-1)
    \Gamma(\chi_k-\omega){\rm d}\omega\\
    &=\sum_{\ell\ge 2}\frac{(-1)^{\ell}(p^{\ell}+q^{\ell})}
    {(\ell-1)!(1-p^{\ell}-q^{\ell})}\,
    \Gamma(\chi_k+\ell-1)\ell(\chi_k+\ell-1).
\end{align*}

Since
\begin{align*}
    G_2(-1+\chi_k)
	&= \mathscr{M}[\tilde{h}_1(z);-1+\chi_k]
	+\mathscr{M}[\tilde{h}_2(z);-1+\chi_k]\\
    &=\mathscr{M}[\tilde{h}_1(z);-1+\chi_k]
	+\Gamma(\chi_k)\left(1-\frac{\chi_k^2+\chi_k+4}
	{2^{\chi_k+2}}\right)+Y(-1+\chi_k),
\end{align*}
we then deduce \eqref{g2k} by collecting all expressions.
\end{document}